\documentclass[twocolumn]{article} 

\usepackage{savesym} \savesymbol{AND}
\usepackage[
  left=1.8cm,
  right=1.8cm,
  top=2.2cm,
  bottom=2.4cm,
  columnsep=0.7cm
]{geometry}

\makeatletter
\renewenvironment{abstract}
  {\noindent\textit{\textbf{Abstract---}}}
  {\par}
\makeatother

\usepackage{amsmath}
\usepackage{amsthm} 
\usepackage{amssymb}
\usepackage{hyperref}
\usepackage{xurl}
\hypersetup{hidelinks=true}

\usepackage{graphicx}
\usepackage[numbers,square]{natbib}
\usepackage{xcolor}
\usepackage{siunitx}
\usepackage{multirow}
\usepackage{comment} 

\usepackage{caption}
\usepackage{subcaption}
\usepackage[algo2e,ruled,vlined,linesnumbered,norelsize]{algorithm2e}
\makeatletter
\newcommand{\removelatexerror}{\let\@latex@error\@gobble}
\makeatother

\newtheorem{theorem}{Theorem}
\newtheorem{lemma}[theorem]{Lemma}
\newtheorem{proposition}[theorem]{Proposition}

\newtheorem{remark}{Remark}
\newtheorem{example}{Example}

\DeclareMathOperator*{\argmin}{arg\,min}

\def\RR{\mathbb{R}}

\def\V{\mathcal{V}}
\def\U{\mathcal{U}}
\def\X{\mathcal{X}}
\def\W{\mathcal{W}}
\def\Co{\mathrm{Co}}
\def\tr{\mathrm{tr}}

\def\bv{{\bf v}}

\def\bX{\boldsymbol{\mathcal{X}}}

\def\ufix{{\bf u}^{\mathrm{fix},n}}
\def\xfix{{\bf x}^{\mathrm{fix},n}}
\def\vfix{{\bf v}^{\mathrm{fix},n}}
\def\Xfix{\bX^{\mathrm{fix},n}}
\def\Jfix{J^{\mathrm{fix},n}}

\def\bpf{\begin{proof}}
\def\epf{\end{proof}}

\allowdisplaybreaks[4]

\title{\textbf{Computationally Tractable Robust Nonlinear Model Predictive Control using DC Programming}\thanks{This work was supported by the EPSRC EP/W524311/1 and by a Schmidt AI in Science Fellowship (M-DS).}}

\author{Martin Doff-Sotta\thanks{Department of Engineering Science, University of Oxford. \texttt{martin.doff-sotta@eng.ox.ac.uk}} \and
Zaheen A-Rahman\thanks{St John’s College / Department of Engineering Science, University of Oxford. \texttt{zaheen.a-rahman@sjc.ox.ac.uk}} \and
Mark Cannon\thanks{Department of Engineering Science, University of Oxford. \texttt{mark.cannon@eng.ox.ac.uk}}}

\date{\today}

\begin{document}
\maketitle

\begin{abstract}
We propose a computationally tractable, robust tube model predictive control (MPC) framework using difference-of-convex (DC) functions and sequential convex programming (SCP). For discrete time systems with a continuous nonlinear state transition map, we show how to construct systematic DC model representations using polynomials and machine learning techniques. We develop a robust tube MPC scheme that convexifies the online optimization by linearising the concave components of constraints, and we provide guarantees of 
recursive feasibility, nominal asymptotic stability, robust constraint satisfaction, and an asymptotic average-cost bound.
We present three data-driven procedures for computing DC models and compare performance using a planar vertical take-off and landing (PVTOL) aircraft case study.
%
\end{abstract}


\section{Introduction}

Robust model predictive control (MPC) is a receding-horizon control strategy that provides robust stability guarantees for systems with uncertain dynamics and state and control constraints \citep{rawlings,kouvaritakis}. Nonlinear system models may be difficult to identify accurately from data, while first-principles models typically contain unknown parameters, disturbances, and structural uncertainty. Uncertainty in the system model makes controller robustness guarantees essential for safe deployment. Achieving such guarantees, however, can lead to conservative designs and substantial computational effort. This paper develops a data-driven modelling and optimization framework for robust nonlinear MPC that improves the trade-off between computational cost and conservatism.

Tube MPC addresses model uncertainty over a future prediction horizon by constructing tubes that bound system trajectories. The approach is well established for linear systems \citep{mayne2005robust,kouvaritakis}, but for nonlinear systems \citep[e.g.][]{mayne2011tube,rawlings,lopez2019dynamic,kohler21,sasfi23},
robust uncertainty propagation remains a major obstacle to deployment. Existing nonlinear tube MPC methods often rely on restrictive structural assumptions 
(such as feedback linearizability, \citet{lopez2019dynamic}), 
computationally intensive offline design (such as global contraction metrics, \citet{kohler21,sasfi23}),
and general-purpose nonconvex solvers for the online MPC problem.

%
%

A common approach for nonlinear MPC for systems with differentiable dynamics is to linearize the model around previously predicted trajectories, yielding a sequence of convex approximations of the online optimization \citep[e.g.][]{biegler98,lee2002constrained}. With suitable bounds on linearization errors \citep[e.g.][]{leeman23}, such methods can provide recursive feasibility and robust stability guarantees using tools from robust linear MPC \citep{mark,buerger26}. Moreover, as sequential convex programming methods, they enjoy at least linear local rates of convergence \citep{messerer21}.
However, the requirement for linearization error bounds can cause conservative approximations \citep{leeman23} or impose restrictive state and input constraints \citep{mark,buerger26}.

When the system model is expressed as a difference of convex functions, the online MPC optimization can be formulated as a difference-of-convex (DC) programming problem \citep{DC-TMPC}. This is a special case of sequential convex programming \citep{messerer21} but it avoids the difficulties of incorporating linearization error bounds in the MPC optimization by exploiting the convex-concave procedure (CCP) \citep{lanckreit09,lipp2016variations}. In particular, the CCP requires linearization of only the concave terms in constraints governing the propagation of tubes over the prediction horizon~\citep{DC-TMPC}, leading to tight global bounds on model approximation errors and hence convergence to a local solution of the nonlinear MPC optimization~\citep{yana}. Recursive feasibility and robust stability guarantees are available in the presence of model uncertainty and disturbances~\citep{thesis, L4DC,yana}.

Although any system model with continuous dynamics can in principle be approximated arbitrarily closely using differences of smooth convex functions \citep{hartman}, first-principles models are not available in DC form except in special cases \citep[e.g.][]{DC-TMPC,yana}. 
Recent work has therefore proposed application-specific data-driven DC modelling approaches combined with physical models, using sum-of-squares convex (SOS-convex) polynomials \citep{doff2023robust, doff2023data}  and neural networks \citep{krausch25,steffen25}.
%
This paper provides a unifying theoretical framework and comparative numerical evaluation of these approaches for robust nonlinear MPC, including additional comparisons with a third method based on sums of radial basis functions. 

The contributions of this paper are as follows: (i) development of a computationally tractable robust MPC algorithm for nonlinear systems with 
a continuous nonlinear state transition map 
and additive disturbances; (ii) presentation of data-driven techniques for deriving system models in DC form; (iii) development of a method to guarantee recursive feasibility in case of additive disturbances; (iv) numerical studies characterizing the performance of each algorithm. 

The paper is organized as follows. Section \ref{sec:DC} introduces data-driven DC modelling.
Section \ref{sec:DC-TMPC} presents the tractable robust nonlinear MPC algorithm and tube parameterization.
Section \ref{sec:dist} extends this to the case of additive disturbances and introduces a backtracking line search to ensure recursive feasibility. 
Section \ref{sec:theory} provides results on convergence, stability, and asymptotic performance.
Section \ref{sec:results} evaluates performance in numerical studies, including a nonconvex PVTOL example.


\textit{Notation:}
At the $n$th discrete instant, the value of $x$ is denoted $x[n]$ and $\{x_{0}, \ldots, x_N\}$ is its predicted trajectory over $N$ steps, where $x_{k}$ is the predicted value of $x[n+k]$.
%
The Euclidean norm is $\|x\|$, and $\lVert x\rVert_P := (x^\top P x)^{1/2}$, where $P=P^\top$ is a positive definite matrix (${P\succ 0}$).
The $l$th component of a function $f:\RR^p\to\RR^q$ is $[f]_l$, and $f$ is a DC function (or simply DC) if $f = g - h$, where $g, h:\RR^p\to\RR^q$ are such that $[g]_l,[h]_l$ are convex functions for $l=1,\ldots,q$.
We denote $\lfloor f^\circ \rfloor(x, u)=f(x^\circ,u^\circ)+A_k (x - x^\circ)  + B_k (u - u^\circ)$ as the Jacobian linearization of a convex function $f$, with $A_k = (\partial f /\partial x)(x^\circ,u^\circ)$, $B_k = (\partial f /\partial u)(x^\circ,u^\circ)$; 
the partial derivatives in this expression can be replaced with subderivatives if $f$ is non-differentiable at $(x^\circ,u^\circ)$.
%
%
%
Optimization operations for vector-valued functions are intended componentwise, so $\max_x f(x)$ has $l$th element $\max_x [f(x)]_l$ for all $l$. 

\section{DC decomposition}
\label{sec:DC}
\subsection{System model}
\label{sec:DC-split}

This work considers nonlinear systems with dynamics
\begin{equation}
\label{eq:sys}
x[n+1] = f(x[n], u[n])
\end{equation}
where $x \in \RR^{n_x}$,  $u \in \RR^{n_u}$ and $f:\RR^{n_x}\times\RR^{n_u}\to\RR^{n_x}$ is continuous.
Any twice continuously-differentiable function $f$ can be expressed in the DC form \citep{hartman} $f = g - h$,
and any continuous function $f$ can be approximated arbitrarily closely by a difference of convex functions on a compact  set~\citep{horst99:dcprog,hartman}. 
%
We describe three methods of computing approximate DC representations; to simplify notation we denote the independent variable as $z$ (e.g.\ $f(z) := f(x[n], u[n])$).

\subsection{SOS-convex polynomials}
\label{sec:polApp}
A DC decomposition can be obtained using  polynomial approximation. We present here a method that applies to a multivariate scalar function to compute the DC decomposition $f = g - h$, where $g, h$ are convex. This method applies to a large class of systems since: a) any continuous function can be approximated arbitrarily closely by a polynomial (Stone-Weierstrass theorem) and; b) any polynomial can be computed as a difference of convex polynomials of same degree \citep{ahmadi}.

\begin{itemize}
    \item[1)]  \textit{Fit polynomial to data.} We assume that the  function\footnote{It is not necessary to have an analytical expression for $f$ but we assume samples of $f$ are available.}  
    $f$ can be approximated arbitrarily closely by a polynomial $\hat{f}$ of degree $2d$ in Gram form such that $\hat{f}(z) = y(z)^\top F y(z)$, where $F = F^\top$ is the Gram matrix and
    $y(z)$ is a vector of monomials in the variable $z$ of degree up to $d$
    ($y$ has size \smash{$C^{|z|}_{d + |z|}$}). Generating $N_s$ samples $f_{s} = f(z^{(s)})$, $s =1, \ldots , N_s$  of the nonlinear model, we solve the following least squares problem: 
\begin{equation*}
\mathcal{LS} : \ \min_{\substack{F}}\  \sum_{s=1}^{N_s} \lVert f_{s} - y(z^{(s)})^\top F y(z^{(s)}) \rVert^2_2 \ \text{ s.t.} \ F = F ^\top. 
\end{equation*}
 \item[2)]  \textit{Compute the Hessians.} Let ${g}(z) = y(z)^\top G y(z)$ and ${h}(z) = y(z)^\top H y(z)$ be convex polynomials\footnote{In particular, they are SOS-convex polynomials \citep{ahmadi}; SOS-convexity implies convexity.} such that their Hessians 
$\nabla^2 g (z)$, $\nabla^2 h (z)$
(with $[\nabla^2 g (z)]_{i,j} = y(z)^\top \mathcal{G}_{ij} y(z)$, 
and $[\nabla^2 h (z)]_{i,j} = y(z)^\top \mathcal{H}_{ij} y(z)$)
are positive semidefinite (PSD). Finding $G, H$ such that $F = G - H$ and ${g}$, ${h}$ are convex reduces to solving a feasibility  semidefinite program (SDP) with constraints $\mathcal{G} \succeq \sigma I  ,   \, \mathcal{H} \succeq \sigma I, \sigma \geq 0$,
 %
%
where $\mathcal{G} := \sum_{i,j} \left(e_ie_j^\top\right) \otimes \mathcal{G}_{ij}$,  $\mathcal{H} := \sum_{i,j} \left(e_ie_j^\top\right) \otimes \mathcal{H}_{ij}$ are block defined by, for all $i, j$,
\begin{align*}
\mathcal{G}_{ij} &= {D}_{ji}^\top G + G {D}_{ij} + {D}_i^\top G D_j  + {D}_j^\top G D_i,
\\
\mathcal{H}_{ij} &= {D}_{ji}^\top (G - F)+ (G - F) {D}_{ij} 
\\
&\quad + {D}_i^\top (G - F) D_j  + {D}_j^\top (G - F) D_i, 
\end{align*}
where $I$ is the identity matrix of compatible dimensions with $i$th column $e_i$, and the coefficient matrices $D_i$ and $D_{ij}=D_i D_j$ are such that $\partial y/ \partial {z_i} = D_i y $ and $\partial^2 y/ \partial {z_i}\partial {z_j} = D_{ij} y $ with $z_i = e_i^\top z$.
\end{itemize}

In general, constructing $\hat{f} = [y^\top F_1 y \; \cdots \; y^\top F_{n_x}y]^\top$, requires computing $G_l,H_l$ for $l=1,\ldots,n_x$ such that
\[
f(z) \approx \underbrace{\begin{bmatrix}
    y(z)^\top G_1 y(z)\\ \vdots \\ y(z)^\top G_{n_x}y(z) 
\end{bmatrix}}_{g(z)} - \underbrace{\begin{bmatrix}
    y(z)^\top H_1 y(z)\\ \vdots \\ y(z)^\top H_{n_x}y(z)
\end{bmatrix}}_{h(z)}. 
\]

\subsection{Input-convex neural networks}

A neural network architecture with convex input-output map was introduced in \citet{amos2017input}. The so-called input-convex neural network (ICNN) is obtained by constraining the kernel weights of the network to be nonnegative during training and using convex activation functions. A fully connected $L$-layer feedforward ICNN with parameters $\theta = \{\Theta_{1:L-1}, \Phi_{0:L-1}, b_{0:L-1}\}$ and input-output map given by $\zeta_L = \mathcal{N}(z; \theta)$ can be defined by the following update equation $\forall l \in \{0, \hdots, L-2 \}$
\[
\zeta_{l+1} = \sigma\big(\Theta_l \zeta_l + \Phi_l z + b_l\big),
\]
with $\Theta_0:= 0$ and $ \zeta_0 := 0$, and the final output layer as
\[
\zeta_L = \Theta_{L-1}\, \zeta_{L-1} + \Phi_{L-1} z + b_{L-1}.
\]
Here $\zeta_l$ is the layer activation, $\Theta_l$ are positively constrained kernel weights ($\Theta_l \geq 0$, for $l = 1, \ldots, L-1 $), $\Phi_l$ are input weights, $b_l$ are bias weights, and $\sigma$ is a convex nondecreasing activation function (e.g.\ ReLU: $\sigma(\zeta) = \max\{\zeta,0\}$). Note that
the input variable $z$ appears explicitly in each layer via skip connections. 

Alternatively, the network can be expressed as a composition of affine transformations $\mathcal{A}_{\theta_{l}, x}(\zeta) = \Theta_l \zeta +  \Phi_l z + b_l$ for $l = 1, \hdots, L-1$, $\mathcal{A}_{\theta_{0}}(z) = \Phi_0 z + b_0$, and the convex nondecreasing activation function $\sigma$:
\begin{equation*}
\mathcal{N}(z ; \theta) = (\mathcal{A}_{\theta_{L-1}, z} \circ \sigma \circ \hdots \circ \sigma \circ \mathcal{A}_{\theta_{1}, z} \circ \sigma \circ \mathcal{A}_{\theta_{0}} ) (z). 
\end{equation*}
Each layer of an ICNN thus consists of the composition of a convex function with a nondecreasing convex function, which ensures that $\mathcal{N}(z;\theta)$ is convex in $z$ (see e.g.~\citet{boyd2004convex}). 

A neural network model with DC structure  \citep{sankaranarayanan2022cdinn} consists of two parallel ICNNs ($\mathcal{N}(z, \theta^g)$, $\mathcal{N}(z, \theta^h)$) whose outputs are subtracted. 
Such a network (which we refer to as DCNN in the sequel) can be trained on data collected from a dynamical system to learn its model in DC form as
\[
f(z) \approx \underbrace{\mathcal{N}(z, \theta^g)}_{g(z)} - \underbrace{\mathcal{N}(z, \theta^h)}_{h(z)}. 
\]

\subsection{Radial basis functions}
A radial basis function (RBF) approximation is an expansion over functions that depend only on the distances between the input $z$ and a set of centres $c_j$, $j=1,\ldots,m$:
\[
f(z) \approx \sum_{j=1}^m \alpha_j \varphi(\| z - c_j\|_{R_j}), 
\]
with parameters $\alpha_j \in \RR^{q}$, $c_j \in \RR^{n}$, and $R_j = R_j^\top \succ 0$ for $j=1,\ldots,m$.
If the RBF kernel $\varphi_j(z) =\varphi(\| z - c_j\|_{R_j})$ is a convex function of $z$, then $f = g - h$ is DC, where
\begin{align*}
  g(z) &= \sum_{j=1}^m \max \{\alpha_j, 0\} \varphi(\| z - c_j\|_{R_j}), 
  \\
  h(z) &= \sum_{j=1}^m \max \{-\alpha_j, 0\} \varphi(\| z - c_j\|_{R_j}). 
\end{align*}
Let $r_j(z) = \| z - c_j\|_{R_j}$, then $\varphi_j(z) = \varphi(r_j(z))$. Since $r_j$ is convex, standard rules for convexity of compositions of functions~\citep[e.g.][]{boyd2004convex} imply that the kernel $\varphi_j$ is convex if $\varphi$ is convex and nondecreasing.



    

An example of a convex RBF kernel is the multiquadric function $\varphi(r) = \sqrt{1 + r^2}$. 
The function $f$ can thus be approximated as a weighted sum of such RBFs as follows
\[
f(z) = \sum_{j=1}^{m} \alpha_j \sqrt{1 + \|z-c_j\|_{R_j}^2 }. 
\]
Various approaches can be used to learn parameters $\alpha_j$, $c_j$, $R_j$, $ j = 1, \hdots, m$ from samples $f_s=f(z^{(s)})$, $s=1,\ldots, N_s$.  
Fixing $c_j$ and $R_j$, the weights $\alpha_j$ can be obtained by solving a least squares problem. Alternatively, general training methods
\citep{hastie2009elements} can be used to optimize all parameters simultaneously.

\section{Successive linearization tube MPC}
\label{sec:DC-TMPC}

In this section, we leverage the data-driven DC decomposition techniques in Section \ref{sec:DC} to develop a robust MPC algorithm for the nonlinear system in equation \eqref{eq:sys}. The method is based on successive linearization and exploits the convex structure in the dynamics to bound the effect of the linear approximation. The MPC optimization is thus formulated as a sequence of computationally tractable convex programs.

Following the tube MPC paradigm, we consider a two-degree-of-freedom control law of the form 
\begin{equation}
\label{eq:u}
u_k = v_k + K_k x_k, 
\end{equation} 
where $K_k$ is a feedback gain and the feedforward control sequence $v_k$  
maintains the model
trajectories within a tube $\{\X_0,\ldots,\X_N\}$. 
The tube cross-sections $\X_k\subset\mathbb{R}^{n_x}$ are polytopic with vertices $\V(\X_k) = \{ \smash{x_k^{(1)}}, \ldots, \smash{x_k^{(\nu)}}\}$  and are parameterized by variables $q_k$ such that
\[
\X_k  = \{ x : \Gamma x \leq q_k \} = \Co \{ x_k^{(1)}, \ldots, x_k^{(\nu)} \} ,
\]
for $k =0, \ldots, N$, where $\Gamma$ is a fixed matrix. In this work we consider explicitly two parameterizations:

\begin{itemize}
    \item \textit{Elementwise bounds.} The components of the model state are constrained by lower and upper bounds $q_{1, k}, {q}_{2, k}$ as: 
$ q_{1, k}\leq x_k \leq q_{2, k}$, $\forall k$. Fixing $\Gamma = [I \ \ {-I}]^\top$ 
    and partitioning $q_k = [q_{2, k}^\top  \ {-{q}_{1, k}^\top}]^\top$
    with $q_{1, k} \leq q_{2, k}$,
    the vertices can be expressed 
    \begin{equation}
    \label{eq:element}
    \V(\X_k) = \bigl\{ x : [x]_l\in\{[q_{1,k}]_l,[q_{2,k}]_l\}, \ l=1,\ldots,n_x \bigr\}.
    \end{equation}
    This parameterization results in $2^{n_x}$ vertices. 
    \item \textit{Simplex.} Let $\Gamma = [{-I} \ \ \underline{1}]^\top$ 
    where $\underline{1}\in\RR^{n_x}$ is a vector of ones. 
    Partitioning $q_k = [\alpha_k^\top \ \beta_k]^\top$, $\alpha_k \in \RR^{n_x}$, $\beta_k \in \RR$, and defining $\sigma_k = \beta_k + \underline{1}^\top \alpha_k\geq 0$, the set of vertices is
     \begin{equation}
     \label{eq:simplex}
 \mathcal{V}(\mathcal{X}_k) = \left\{-\alpha_k, -\alpha_k  + e_1 \sigma_k, \dots, -\alpha_k  + e_{n_x} \sigma_k \right\},
 \end{equation}
where $e_1, \ldots, e_{n_x}$ are standard basis vectors of $\mathbb{R}^{n_x}$. This parameterization results in ${n_x}+1$ vertices. 
\end{itemize}

Let $f_\Gamma = \Gamma f $. Then $x_{k+1}\in\X_{k+1}$ for all $x_k\in\X_k$ if and only if, for all $x_k \in \X_k$, we have
\begin{equation}
\label{eq:f_Q}
f_\Gamma(x_k , v_k + K_k x_k) \leq q_{k+1}.
\end{equation}
%
%
To convexify the MPC optimization, we use DC decomposition\footnote{In this section we ignore errors due to DC approximation methods such as those discussed in Section~\ref{sec:DC}; approximation and modelling errors are accounted for in Section~\ref{sec:dist}.} to express $f_\Gamma = g_\Gamma - h_\Gamma$ where $g_\Gamma,h_\Gamma$ defined
\[
g_\Gamma=\Gamma^+ g + \Gamma^- h, \ \ 
h_\Gamma= \Gamma^+ h + \Gamma^- g ,
\]
are convex, with $\Gamma = \Gamma^+ -\Gamma^-$ and $\Gamma^+, \Gamma^- \geq 0$.

Inequality \eqref{eq:f_Q} can then be relaxed by linearizing the concave part of the DC decomposition about  a predicted trajectory $({\bf x}^\circ, {\bf u}^\circ)$ as follows
\begin{equation}
\label{eq:f_Q_hat}
 g_\Gamma( x_k, v_k + K_k x_k) - \lfloor h_\Gamma ^\circ \rfloor(x_k, v_k + K_k x_k) \leq q_{k+1}.
\end{equation}
By convexity we have the global bound $\lfloor h_\Gamma ^\circ \rfloor \leq h_\Gamma$,
so \eqref{eq:f_Q_hat} is sufficient for \eqref{eq:f_Q}. Crucially, \eqref{eq:f_Q_hat} is convex in $v_k,x_k,q_{k+1}$. 

In addition to the dynamical constraint \eqref{eq:f_Q_hat}, state and control constraints of the form, $\forall k \in \{0, \ldots, N-1\}$%
\begin{equation}\label{eq:ctr}%
    x_0= x[n], \ \ x_k \in \X, \ \ u_k=v_k + K_k x_k \in \U,  \ \ x_N  \in \hat{\X}
\end{equation}%
are assumed, where $\X, \hat{\X}, \U$ are compact convex sets and $x_0$ is a given initial state. 
The feedforward control sequence  $\bv = v_{0:N-1}$  
and the parameters ${\bf q} =q_{0:N}$ defining the
tube $\bX = \X_{0:N}$ 
are variables in an optimal control problem with constraints \eqref{eq:f_Q_hat}-\eqref{eq:ctr} and cost
\begin{multline}
\label{eq:obj0}
J(\bX,\bv) = \sum_{k=0}^{N-1} \max_{x \in \mathcal{X}_k} \{ \| x - x^r \|^2_Q  + \| v_k + K_k x  - u^r \|^2_R \}
\\
+ \max_{x \in \mathcal{X}_N} \| x - x^r \|^2_{\hat{Q}} .
\end{multline}
Here $(x^r,u^r)\in\hat{\X}\times\U$ define a target equilibrium point of (\ref{eq:sys}), and $R , Q\succ 0$.
To ensure  asymptotic stability,
we assume\footnote{See e.g.~\citet{L4DC} for how to compute $\hat{K},\hat{X},\hat{Q}$.} the terminal set $\smash{\hat{\X}}$ is positively invariant under a terminal control law $u=\smash{\hat{K}} (x-x^r)+u^r$, i.e.~$\forall x\in\smash{\hat{\X}}$
\begin{equation}\label{eq:invariance}
f(x,\hat{K} (x-x^r)+u^r) \in \hat{\X}\subseteq \X, \ \ \hat{K}(x-x^r)+u^r \in \U,
\end{equation}
%
%
and the terminal weight $\hat{Q} \succ 0$ satisfies, $\forall x\in\hat{\X}$
\begin{multline}\label{eq:Q_N}
\lVert x-x^r\rVert_{\hat{Q}}^2 \geq 
\lVert f(x,\hat{K}(x - x^r)+u^r) - x^r\rVert_{\hat{Q}}^2 
\\
+ \lVert x - x^r \rVert_Q^2
+ \lVert \hat{K}( x - x^r) \rVert_R^2.
\end{multline}

The maximization in~(\ref{eq:obj0}) is implemented using the vertices of $\bX$ by
minimizing  $\|\theta\|^2$, where $\theta = \theta_{0:N}$,
subject to, $\forall k = 0, \ldots, N\!-\!1$, $\forall\hat{x}_k \!\in\!\V(\X_k)$, $\forall\hat{x}_N \!\in\!\V(\X_N)$,
\begin{align}
(\| \hat{x}_k  - x^r\|_Q^2 +  \|v_k + K_k \hat{x}_k - u^r\|_R^2)^{1/2} &\leq \theta_k, 
\label{eq:obj}
\\
\| \hat{x}_N - x^r \|_{\hat{Q}} &\leq \theta_N . 
%
\end{align}
%
%
The dynamic constraint in \eqref{eq:f_Q_hat} can also be reformulated in terms of the vertices of the tube $\bX$.  By convexity of $g_\Gamma$ and linearity of $\lfloor h_\Gamma^\circ \rfloor$, \eqref{eq:f_Q_hat} is equivalent to, $\forall\hat{x}_k \in\V(\X_k)$,
\begin{equation}
\label{eq:new_dyna}
g_\Gamma(\hat{x}_k, v_k + K_k \hat{x}_k) - \lfloor h_\Gamma ^\circ \rfloor(\hat{x}_k, v_k + K_k \hat{x}_k) \leq q_{k+1}.
\end{equation}
%
%
Gathering \eqref{eq:ctr}, \eqref{eq:obj}-\eqref{eq:new_dyna}, and exploiting the convexity of $\X, 
\smash{\hat{\X}}, \U$, the MPC optimization is a convex program:%
\begin{equation}\label{eq:cvx}%
\begin{aligned}
& ({\bf v}^\ast,\theta^\ast,{\bf q}^\ast) = \argmin_{\bv,\theta, {\bf q}} \ \|\theta\|^2
\\
& \text{s.t., for $k=0, \ldots, N-1$, 
$\forall \hat{x}_k \in\V(\X_k)$, $\forall \hat{x}_N \in\V(\X_N)$:}
\\ 
&  (\| \hat{x}_k  - x^r \|^2_Q  + \| v_k + K_k \hat{x}_k - u^r \|^2_R)^{1/2} \leq \theta_k 
\\
& \| \hat{x}_N - x^r \|_{\hat{Q}} \leq  \theta_N
\\
&  g_\Gamma(\hat{x}_k, v_k + K_k \hat{x}_k) - \lfloor h_\Gamma ^\circ \rfloor(\hat{x}_k, v_k + K_k \hat{x}_k) \leq q_{k+1}
\\
& x[n] =\hat{x}_0 , \quad \hat{x}_k \in \X, 
\quad v_k + K_k \hat{x}_k \in \U, 
\quad \hat{x}_N\in\hat{\X}
\end{aligned}
\end{equation}
where $\V(\X_k)$ is defined in terms of the variable $q_{k}$ by \eqref{eq:element} or  \eqref{eq:simplex}, depending on the tube parameterization. 

After solving~(\ref{eq:cvx}), the linearization trajectory $({\bf x}^\circ,{\bf u}^\circ)$ is updated by setting ${\bf v}^\circ \gets {\bf v}^\ast$ 
and for $k = 0, \ldots , N-1$
\begin{equation}\label{eq:iter1}%
u^\circ_k = v_k^\circ + K_k x_k^\circ, 
\quad
x^\circ_{k+1} = f(x^\circ_k, u^\circ_k). 
\end{equation}
The iteration (\ref{eq:cvx})-(\ref{eq:iter1}) is repeated until the change in objective falls below a threshold or a specified maximum number of iterations is reached.
The control law is then implemented as $u[n] = u^\circ_{0}$, and at the subsequent time instant (i.e.\ at time $n+1$) we set $x^\circ_0\gets x[n+1]$ and
\begin{equation}\label{eq:step_update3}
\begin{gathered}
K_k \gets K_{k+1}, \ \ k=0,\ldots, N-2, \ \ K_{N-1} = \hat{K}, 
\\ 
{\bv}^\circ \gets \{ v_1^\circ, \ldots, v_{N-1}^\circ, u^r-\hat{K}x^r \}, 
\end{gathered}
\end{equation}
and then recompute $({\bf x}^\circ,{\bf u}^\circ)$ using (\ref{eq:iter1}).
These updates ensure recursive feasibility of problem~(\ref{eq:cvx}), as we show in Section~\ref{sec:theory}, and the optimal cost
\[
J({\bX}^\ast, {\bv}^\ast) = \|\theta^\ast\|^2
\]
is non-increasing at successive iterations and hence converges asymptotically. This implies that the iteration~(\ref{eq:cvx})-(\ref{eq:iter1}) converges to a fixed point, $\bv^\ast \to {\bf v}^\circ$ and $\bX^\ast \to \{\{x^\circ_0\},\ldots,\{x^\circ_N\}\} $, and, under mild regularity assumptions, this fixed point satisfies the first-order necessary conditions for the
problem of minimizing~(\ref{eq:obj0}) subject to (\ref{eq:f_Q}) and (\ref{eq:ctr}) (as discussed in Section~\ref{sec:theory}).
%
The iteration is summarized in Algorithm~\ref{algo6}. 

\begin{remark}
Convexity of (\ref{eq:cvx}) and compactness of $\X,\U$ ensure that a solution $({\bf v}^\ast,\theta^\ast,{\bf q}^\ast)$ exists if (\ref{eq:cvx}) is feasible, and in this case ${\bf v}^\ast$ is unique since (\ref{eq:obj0}) is strictly convex in ${\bf v}$. If the solution for ${\bf q}^\ast$ is non-unique, we assume a unique optimal tube $\bX^\ast$ is selected from the solution set.  
\end{remark}

\begin{remark}
A tube parameterization in terms of elementwise bounds forces the number of constraints in \eqref{eq:cvx} to depend exponentially on the state dimension $n_x$.
However, knowledge of the functions $g_\Gamma$, $\lfloor h_\Gamma^\circ\rfloor$ and the gain matrices $K_{0:N-1}$ can often be exploited to reduce the number of constraints by identifying vertices that may be active \textit{a priori}. Alternatively, if simplex tubes are used, then the number of constraints in~(\ref{eq:cvx}) is linear in $n_x$.
\end{remark}

\begin{remark}\label{rem:initialization}
Algorithm~\ref{algo6} assumes knowledge of a feasible trajectory at initial time. 
%
%
Given a trajectory $({\bf x}^\circ, {\bf u}^\circ)$ of~(\ref{eq:sys}) with ${(x^\circ_k,u^\circ_k)\in\X\times\U}$, $k=0,\ldots,N-1$ and $x_N^\circ\in\smash{\hat{\X}}$ but $x^\circ_0 \neq x[0]$,
computed for example by setting $u^\circ_{0:N-1} = \smash{\hat{K}}(x_{0:N-1}^\circ - x^r) + u^r$ with $x_0^\circ\in\smash{\hat{\X}}$,
an SCP-based approach is to relax the 
initial constraint in (\ref{eq:cvx}) to $x[0] + \tilde{x} = \hat{x}_0$,
where  $\tilde{x}$ is a variable, and to replace the objective of (\ref{eq:cvx}) with $\|\tilde{x}\|^2$. 
Repeatedly solving for the optimal $\tilde{x}^\ast,{\bf v}^\ast$ and updating $x^\circ_0 \gets x[0] + \tilde{x}^\ast$ and
$({\bf x}^\circ, {\bf u}^\circ)$ via (\ref{eq:iter1}) yields an iteration with non-increasing $\|x[0]-x_0^\circ\|$.
%
\end{remark}

\removelatexerror
\begin{algorithm2e}[ht]
\caption{Successive linearization MPC}
\label{algo6}
\DontPrintSemicolon
\SetKwBlock{kwInit}{Initialization}{}
\SetKwBlock{kwOnline}{Online at times $n=0,1,\ldots:$}{}
    \kwInit{
      compute terminal parameters $\hat{K}$, $\hat{Q}$, $\hat{\X}$
    }
    \kwOnline{
      set $x_0^\circ \gets x[n]$
      \\
      \eIf{$n = 0$}{
        find $({\bf x}^\circ, {\bf u}^\circ)$ such that $x^\circ_{k+1} = f(x^\circ_k,u^\circ_k)$, $(x_k^\circ, u_k^\circ)\in\X\times\U$, $k=0,\ldots,N-1$, $x_N^\circ\in\hat{\X}$
       \\
       compute local feedback gains $K_{0:N-1}$
    }{
     compute $K_{0:N-1},{\bf v}^\circ,({\bf x}^\circ, {\bf u}^\circ)$ using (\ref{eq:step_update3}), (\ref{eq:iter1})
    }
    set $\smash{J^{0,n}} \gets {\infty}$, $j \gets 0$
    \\
    \Repeat{$\smash{J^{j-1,n}} - \smash{J^{j,n}} < $ tolerance \textnormal{\textbf{or}} $j = $ itermax}{
    linearize $h_\Gamma$ around $({\bf x}^\circ, {\bf u}^\circ)$ to obtain $\lfloor h_\Gamma^\circ \rfloor$\\
    solve (\ref{eq:cvx}) to find the optimal ${\bv}^\ast$ and $\bX^\ast$ \\
    update the iteration counter: $j \gets j+1$\\
    store the objective:
    $\smash{J^{j,n}} \gets \|\theta^\ast\|^2$\\
    set ${\bf v}^\circ \gets {\bv}^\ast$ and update $({\bf x}^\circ, {\bf u}^\circ)$ using (\ref{eq:iter1})\\
    }
implement the control input: $u[n] \gets u^\circ_0$\\
}
\end{algorithm2e}

\section{Additive disturbances}
\label{sec:dist}

Bounded additive disturbances in the system model 
%
\begin{equation}
\label{eq:disturbance}
x[n+1] = g(x[n], u[n]) - h(x[n], u[n]) + w[n]
\end{equation}
can be accounted for using a robust tube.
Here $g, h$ are convex, $w\in\mathcal{W}$ is an unknown disturbance, and $\mathcal{W}$ is a known compact set.   
If the DC representation of $f$ has a non-zero approximation error, then the disturbance set has the form ${\mathcal{W} = \overline{\mathcal{W}} \oplus \mathcal{E}}$, where $\overline{\mathcal{W}}$ and $\mathcal{E}$ respectively bound external disturbances and approximation errors.

To account for the disturbance in (\ref{eq:disturbance}), the terminal conditions on $\smash{\hat{K}}, \smash{\hat{\X}}, \smash{\hat{Q}}$ in (\ref{eq:invariance}) and (\ref{eq:Q_N}) must be strengthened. In this case we require, for some $\beta \geq 0$,
$\forall x\in\smash{\hat{\X}}$, $\forall w\in\W$,
\begin{align}
& g(x, \hat{u}(x)) - h(x, \hat{u}(x)) + w \in \hat{\X}\subseteq \X, \ \ \hat{u}(x) \in \U,
\label{eq:robust_invariance}
\\
&
\lVert x-x^r\rVert_{\hat{Q}}^2 \geq 
\lVert g(x,\hat{u}(x)) - h(x,\hat{u}(x)) + w - x^r\rVert_{\hat{Q}}^2 
\notag\\
&\qquad\qquad\qquad + \lVert x - x^r \rVert_Q^2 + \lVert \hat{u}( x) - u^r\rVert_R^2  - \beta,
\label{eq:Q_Nbis}
\end{align}
where $\hat{u}(x) = \smash{\hat{K}}(x-x^r)+u^r$.


In the presence of unknown future disturbances, the predicted trajectories of  \eqref{eq:disturbance} under a given control law are uncertain. Hence the tube $\bX$ bounding predicted states cannot converge to a single trajectory as in the undisturbed case. 
The feasibility of successive convexified MPC optimizations can be ensured by careful design of the linearization trajectory $({\bf x}^\circ,{\bf u}^\circ)$~\citep[e.g.][]{yana}.
However, we propose a simpler alternative that linearizes around trajectories of the nominal model and uses a line search to recover from infeasibility.

The algorithm retains an initial state $\bar{x}_0^\circ$ and control perturbation sequence $\bar{{\bf v}}^\circ$ such that the trajectory $({\bf x}^\circ,{\bf u}^\circ)$ generated by the nominal model, for $k=0,\ldots,N-1$,
\begin{equation}\label{eq:nominal_model}
u_k^\circ = v_k^\circ + K_kx_k^\circ, \quad x_{k+1}^\circ = g(x_k^\circ, u_k^\circ) - h(x_k^\circ, u_k^\circ), 
\end{equation} 
necessarily generates a feasible MPC optimization problem when $(x^\circ_0,{\bf v}^\circ) = (\bar{x}^\circ_0, \bar{{\bf v}}^\circ)$.
To realise this, after each successful solve of the MPC optimization, we set
$(\bar{x}^\circ_0,\bar{\bf v}^\circ) \gets (x^\circ_0, {\bf v}^\circ)$,
${\bf v}^\circ\gets {\bf v}^\ast$,
and recompute the linearization trajectory $({\bf x}^\circ,{\bf u}^\circ)$ using~(\ref{eq:nominal_model}). 
At the first iteration of the subsequent time instant (at time $n+1$) we 
set  $x_0^\circ \gets x[n+1]$, 
$\bar{x}^\circ_0 \gets \bar{x}_1^\circ$, 
$v^r = u^r - \smash{\hat{K}}x^r$, and%
\begin{equation}\label{eq:v_update}%
\begin{gathered}
K_k \gets K_{k+1}, \ \ k=0,\ldots, N-2, \ \ K_{N-1} = \smash{\hat{K}},
\\ 
{\bv}^\circ \gets \{ v_1^\circ, \ldots, v_{N-1}^\circ, v^r\}, \ \ 
\bar{\bv}^\circ \gets \{ \bar{v}_1^\circ, \ldots, \bar{v}_{N-1}^\circ, v^r\}, 
\end{gathered}
\end{equation}
and recompute $({\bf x}^\circ,{\bf u}^\circ)$ using (\ref{eq:nominal_model}).
If at any iteration the optimization is infeasible, 
then the following interpolation is performed,
\begin{equation}
\label{eq:line_search}
(x_0^\circ,{\bf v}^\circ) \gets \rho (x_0^\circ,{\bf v}^\circ) + (1-\rho) (\bar{x}_0^\circ,\bar{{\bf v}}^\circ), \ \ \rho\in [0,1);
\end{equation}
then the linearization trajectory $({\bf x}^\circ,{\bf u}^\circ)$ is recomputed using~(\ref{eq:nominal_model}) and the optimization re-attempted.
This process is repeated until a feasible problem is found or the maximum number of line search iterations is exceeded. In the latter case the feasible fall-back $(\bar{x}^\circ_0, \bar{{\bf v}}^\circ)$ is used.

The following convex program replaces the MPC optimization \eqref{eq:cvx} when bounded disturbances are present:
\begin{equation}\label{eq:cvx_dist}
\begin{aligned}
& ({\bf v}^\ast,\theta^\ast,{\bf q}^\ast) = \argmin_{\bv,\theta, {\bf q}} \ \|\theta\|^2
\\
& \text{s.t., for $k = 0, \ldots, N-1$, $\forall \hat{x}_k \in\V(\X_k)$, $\forall \hat{x}_N \in\V(\X_N)$:}
\\ 
&  (\| \hat{x}_k  - x^r \|^2_Q  + \| v_k + K_k \hat{x}_k - u^r \|^2_R)^{1/2} \leq \theta_k 
\\
& \| \hat{x}_N - x^r \|_{\hat{Q}} \leq  \theta_N
\\
& \|\theta\| \leq \gamma
\\
& g_\Gamma(\hat{x}_k, v_k + K_k \hat{x}_k) - \lfloor h_\Gamma ^\circ \rfloor(\hat{x}_k, v_k + K_k \hat{x}_k) + \bar{w} \leq q_{k+1}
\\
& x[n] \in \X_0, \quad \hat{x}_k \in \X,
\quad v_k + K_k \hat{x}_k \in \U, 
\quad \hat{x}_N\in \hat{\X}
\end{aligned}
\end{equation}
where the cost bound $\gamma$ is adjusted online to ensure convergence, 
$g_\Gamma,h_\Gamma$ are defined so that $\Gamma f = f_\Gamma= g_\Gamma - h_\Gamma$, and
$\bar{w} = \max_{w \in \mathcal{W}} \Gamma w$.
%
%
%
The successive linearization robust tube MPC is summarized in Algorithm \ref{algo7}.


\begin{algorithm2e}[h!]
\caption{Successive linearization robust MPC}
\label{algo7}
\DontPrintSemicolon
\SetKwBlock{kwInit}{Initialization}{}
\SetKwBlock{kwOnline}{Online at times $n=0,1,\ldots:$}{}
    \kwInit{
      compute $\hat{K}$, $\hat{Q}$, $\hat{\X}$; set $\gamma \gets \infty$
    }
    \kwOnline{
      set $x_0^\circ \gets x[n]$
      \\
      \eIf{$n = 0$}{
         find $({\bf x}^\circ, {\bf u}^\circ)$ s.t. $x^\circ_{k+1} = g(x^\circ_k,u^\circ_k) - h(x^\circ_k,u^\circ_k)$, $(x_k^\circ, u_k^\circ)\in\X\times\U$, $k=0,\ldots,N-1$, $x_N^\circ\in\hat{\X}$\\
        compute local feedback gains $K_{0:N-1}$ and verify that  (\ref{eq:cvx_dist}) is feasible\\
        set $v^\circ_{k} \gets u^\circ_k - K_kx^\circ_k$, $k=0,\ldots,N-1$\\
    }{
    set $\bar{x}_0^\circ \gets \bar{x}_1^\circ$, update $K_{0:N-1}, {\bf v}^\circ, \bar{\bf v}^\circ$ using (\ref{eq:v_update})\!\!\\
    compute $({\bf x}^\circ, {\bf u}^\circ)$ using (\ref{eq:nominal_model})\\
    }
    set $\smash{J^{0,n}} \gets {\infty}$, $j \gets 0$
    \\    
    \Repeat{$\smash{J^{j-1,n}} - \smash{J^{j,n}} < $ tolerance \textnormal{\textbf{or}} $j =$ itermax}{
    linearize $h_\Gamma$ around $({\bf x}^\circ, {\bf u}^\circ)$ to obtain $\lfloor h_\Gamma^\circ \rfloor$\\
    solve (\ref{eq:cvx_dist}) to find the optimal $\bv^\ast$ \\
    set $l \gets 0$\\
    \While{(\ref{eq:cvx_dist}) infeasible \textnormal{\textbf{and}} $l \leq$ itermax\_ls}{
    \eIf{$l <$ itermax\_ls}{
    update $(x_0^\circ, {\bf v}^\circ)$ via \eqref{eq:line_search} for $\rho\in(0,1)$
    }{
    set $(x_0^\circ, {\bf v}^\circ) \gets (\bar{x}_0^\circ, \bar{\bf v}^\circ)$}  
    update $({\bf x}^\circ, {\bf u}^\circ)$ using (\ref{eq:nominal_model})\\
    linearize $h_\Gamma$ around $({\bf x}^\circ, {\bf u}^\circ)$ to obtain $\lfloor h_\Gamma^\circ \rfloor\!\!$\\
    solve (\ref{eq:cvx_dist}) to find the optimal $\bv^\ast$\\
    update the iteration counter: $l\gets l+1$
    }
    update the iteration counter: $j \gets j+1$\\
    store the objective: $\smash{J^{j,n}} \gets \smash{\|\theta^\ast\|^2}$\\
    set $(\bar{x}^\circ_0,\bar{\bf v}^\circ) \gets (x^\circ_0, {\bf v}^\circ)$,\\
    set ${\bf v}^\circ\gets {\bf v}^\ast$ and update $({\bf x}^\circ, {\bf u}^\circ)$ using (\ref{eq:nominal_model})
    }
implement the control input: $u[n] \gets v_0^\circ + K_0 x[n]$\\
set $\gamma^2 \gets J^{j,n} - \|x[n]-x^r\|_Q^2 - \|u[n] - u^r\|_R^2 + \beta$
}
\end{algorithm2e}

\begin{remark}
For simplicity, in Algorithms~\ref{algo6} and \ref{algo7} the local feedback gains $K_{0:N-1}$ are updated only  once per discrete time-step, 
in Alg.~\ref{algo6} line 7 or 9, and
Alg.~\ref{algo7} line 7 or 10.
For faster convergence, $K_{0:N-1}$ may be recomputed immediately after each update of 
$({\bf x}^\circ,{\bf u}^\circ)$ in Alg.~\ref{algo6} line 9, and 
in Alg.~\ref{algo7} lines 11, 22, and 29. 
In this case the stored ${\bf v}^\circ$ and $\bar{\bf v}^\circ$ trajectories must be updated by setting $K^{\mathrm{old}}_{0:N-1}\gets K_{0:N-1}$ before updating $K_{0:N-1}$, and setting  $v^\circ_k\gets v^\circ_k + (K^{\mathrm{old}}_k-K_k)x^\circ_k$, $\bar{v}^\circ_k\gets \bar{v}^\circ_k + (K^{\mathrm{old}}_k-K_k)\bar{x}^\circ_k$, $k=0,\ldots,N-1$ after updating $K_{0:N-1}$.
\end{remark}

\section{Feasibility, optimality, and stability}
\label{sec:theory}

%
We first show that the convexified optimization problems solved by Algorithms~\ref{algo6} and \ref{algo7} are recursively feasible, meaning that they are feasible at all iterations if initially feasible.
This implies that the successive linearization process (in lines 11-17 of Algorithm~\ref{algo6} and lines 13-30 of Algorithm~\ref{algo7}) can be terminated after any number of iterations at each discrete time-step without compromising either feasibility or, as shown later in this section, stability and constraint satisfaction. Hence the approach allows early termination of the MPC optimization.

\begin{proposition}[Feasibility]\label{prop:feas}
Assume the linearization trajectory $({\bf x}^\circ,{\bf u}^\circ)$ at initial time generates a feasible instance of problem~(\ref{eq:cvx}) in Alg.~\ref{algo6} line 13, [or problem (\ref{eq:cvx_dist}) in Alg.~\ref{algo7} line 15].
Then, at all times $n\geq 0$ and iterations $j\geq 0$, (\ref{eq:cvx}) in Alg.~\ref{algo6} line 13, [or (\ref{eq:cvx_dist}) in Alg.~\ref{algo7}, either at line 15 or line 24 for some iteration $l$] will be feasible.
\end{proposition}

\bpf
For Alg.~\ref{algo6} with no disturbances, recursive feasibility follows from: (i) the update of $({\bf x}^\circ, {\bf u}^\circ)$ in lines 8 and 16 together with the invariance property of $\smash{\hat{\X}}$ in (\ref{eq:invariance}); (ii) the observation that, if $({\bf x}^\circ, {\bf u}^\circ)$ satisfies $(x^\circ_k,u^\circ_k)\in\X\times\U$ for $k=0,\ldots,{N-1}$ and $x_N^\circ\in\smash{\hat{\X}}$, then problem (\ref{eq:cvx}) is feasible since  the constraints of (\ref{eq:cvx}) are trivially satisfied by $(\bX,{\bf v})$ with ${\bf v} = {\bf v}^\circ$ and $\X_k = \{x^\circ_k\}$ for all $k$.

For Alg.~\ref{algo7} with disturbances, 
if (\ref{eq:cvx_dist}) in line 15 is infeasible at the first iteration of a given time step, then  
the line search (lines 17-25) necessarily finds $({\bf x}^\circ,{\bf u}^\circ)$ generating a feasible instance of (\ref{eq:cvx_dist}) 
due to
the method of updating $(\bar{x}^\circ_0,\bar{\bf v}^\circ)$ (line 28), the definition of $\gamma$ (line 32), and the terminal conditions (\ref{eq:robust_invariance}),(\ref{eq:Q_Nbis}).
The definition of $(\bar{x}^\circ_0,\bar{\bf v}^\circ)$ (line 28) ensures that the line search recovers feasibility of (\ref{eq:cvx_dist}) at all subsequent linearization iterations.
%
\epf


In the disturbance-free case, the iteration performed at a given time-step $n$ in Alg.~\ref{algo6} lines 11-17 converges to a fixed point, which we denote $(\Xfix, \vfix)$ with linearization trajectory $(\xfix, \ufix)$ and corresponding cost $\Jfix = J (\Xfix ,\vfix)$. This can be shown as follows.

\begin{lemma}[Cost convergence]
\label{lem:cost_convergence}
At successive iterations of Alg.~\ref{algo6} at discrete time $n$, 
the objective $J^{j,n}$ of (\ref{eq:cvx}) in line 13 satisfies $J^{j+1,n} \leq J^{j,n}$ for all $j\geq 0$ and converges to a limit, $J^{j,n}\to\Jfix$ as $j\to\infty$.
\end{lemma}

\bpf
We must have $J^{j+1,n} \leq J^{j,n}$ for all $j\geq 0$ due to the recursive feasibility of problem (\ref{eq:cvx}) in Alg.~\ref{algo6},
since $J^{j+1,n}$ can be no greater than the cost of the feasible solution with ${\bf v} = {\bf v}^\circ$ and $\bX = \bigl\{\{x^\circ_0\},\ldots,\{x^\circ_N\}\bigr\}$, which can be no greater than $J^{j,n}$.
Since $J^{j,n}\geq 0$, this implies $J^{j,n}\to \Jfix$ as $j\to\infty$. 
\epf

\begin{proposition}[Convergence]
\label{prop:convergence}
Successive iterations of Alg.~\ref{algo6} at a given discrete time $n$
converge to a fixed point, with $(\bX^\ast,{\bf v}^\ast)\to (\Xfix,\vfix)$ in line 13, and 
$({\bf x}^\circ,{\bf u}^\circ)\to (\xfix,\ufix)$ in line 16.
%
\end{proposition}

\bpf
Lemma~\ref{lem:cost_convergence} gives $J^{j+1,n}-J^{j,n}\to 0$ as $j\to\infty$. 
But $({\bf x}^\circ, {\bf u}^\circ)$ is feasible for problem (\ref{eq:cvx}) at iteration $j+1$, so if $J^{j+1,n} = J^{j,n}$, then ${\bf v}^\ast={\bf v}^\circ$, and this is the unique minimizer since $J(\bX,{\bf v})$ is strictly convex in ${\bf v}$. 
Therefore $J^{j+1,n}-J^{j,n}\to 0$ implies ${\bf v}^\ast\to \vfix$ due to the convexity and continuity of $J(\bX,{\bf v})$ in ${\bf v}$.
In addition, from (\ref{eq:iter1}), $({\bf x}^\circ, {\bf u}^\circ)$ depends continuously on ${\bf v}^\circ$ and it follows that $({\bf x}^\circ,{\bf u}^\circ)\to (\xfix,\ufix)$
and $\bX^\ast \to \Xfix = \bigl\{ \{ x_0^{\mathrm{fix},n} \},\ldots, \{ x_N^{\mathrm{fix},n} \} \bigr\}$.
\epf

To consider how a fixed point of the iteration in Alg.~\ref{algo6} relates to an  underlying optimal control problem, let
\begin{align}
({\bf x}^\ast,{\bf u}^\ast) = \argmin_{x_{0:N},\, u_{0:N-1}}
& \sum_{k=0}^{N-1} (\|x_k\!-\!x^r\|_Q^2 + \|u_k\!-\!u^r\|_R^2) 
\notag \\
& + \|x_N-x^r\|_{\hat{Q}}^2
\label{eq:nlp}
\end{align}
subject to $x_{k+1} = f(x_k,u_k)$ and the constraints (\ref{eq:ctr}). Assume $f$, $g_\Gamma$ and $h_\Gamma$ are smooth and that (\ref{eq:nlp}) satisfies the linear independence constraint qualification (LICQ) at $({\bf x}^\ast,{\bf u}^\ast)$. Then it can be shown
(e.g.~using the approach of Theorem~3 of \citet{yana} or Lemma~4.1 of \citet{messerer21})
that $(\xfix,\ufix)$ satisfies the first order necessary optimality conditions for problem~(\ref{eq:nlp}) and any local optimum $({\bf x}^\ast,{\bf u}^\ast)$ of (\ref{eq:nlp}) is a fixed point $(\xfix,\ufix)$.





We next show that under Algorithm \ref{algo6}, the closed loop system is asymptotically stable in the undisturbed case.

\begin{theorem}[Nominal stability] 
\label{th:AS}
Under the control law of Algorithm~\ref{algo6}, $(x^r,u^r)$ is an asymptotically stable equilibrium of the undisturbed system~(\ref{eq:sys}) with region of attraction consisting of all initial conditions such that (\ref{eq:cvx}) is feasible for some $({\bf x}^\circ,{\bf u}^\circ)$.
\end{theorem}

\bpf
Denote the objective of~(\ref{eq:cvx}) computed in Alg.~\ref{algo6} line 13 at the final iteration at the $n$th discrete time-step as $J^{j_n,n}$.
Then at the first iteration at time $n+1$, 
the update of $({\bf x}^\circ,{\bf u}^\circ)$ in line 9 and condition~(\ref{eq:invariance}) 
imply that a feasible solution of~(\ref{eq:cvx}) is given by 
${\bf v} = {\bf v}^\circ$ and 
$\bX = \{\{x_0^\circ\},\ldots,\{x_N^\circ\}\}$.
From (\ref{eq:Q_N}) and since $J^{j+1,n} \leq J^{j,n}$ for all $j\geq 0$, we conclude that 
the optimal cost computed at the final iteration is non-increasing with $n$:
%
\begin{equation}\label{eq:cost_decrease}
J^{j_{n+1},n+1}\leq J^{1,n+1}\leq J^{j_n,n} - \|x[n] - x^r\|_Q^2 - \|u[n] - u^r\|_R^2.
\end{equation}
Furthermore, $\underline{\lambda}\|x[n] - x^r\|^2 \leq J^{j_n,n}\leq \bar{\lambda}\|x[n] - x^r\|^2$, where
$\underline{\lambda}$ is the smallest eigenvalue of $Q$, and 
the existence of $\bar{\lambda}$ follows from 
continuity of $J^{j_n,n}$ in $x[n]$ and (\ref{eq:invariance}), (\ref{eq:Q_N}), which 
imply a quadratic bound on the cost associated with the terminal control law when $x[n]$ belongs to the interior of $\smash{\hat{\X}}$.
Therefore $(x,u)=(x^r,u^r)$ is asymptotically stable and 
the performance of the closed-loop system is bounded, for any positive integer $t$, by 
\[
\sum_{n=0}^{t-1} \bigl(\|x[n] - x^r\|_Q^2 + \|u[n] - u^r\|_R^2\bigr) \leq J^{j_0,0} .
\]
\epf

For the system~(\ref{eq:disturbance}) with bounded disturbances, the control law of Algorithm~\ref{algo7} ensures robust constraint satisfaction and boundedness of the optimal cost of problem~\ref{eq:cvx_dist} as a consequence of Proposition~\ref{prop:feas} and the requirement that $\X,\U$ are compact. In addition, we have the following asymptotic performance bound.




\begin{theorem}[Robust average cost bound]
\label{th:quad_stability}
Under Algorithm~\ref{algo7}, the uncertain system \eqref{eq:disturbance} satisfies the following asymptotic bound on the time-average stage cost:
\begin{equation}\label{eq:stagecost_bound}
\limsup_{t \rightarrow \infty} \frac{1}{t} \sum_{n=0}^{t-1}  (\|x[n] - x^r\|^2_Q + \|u[n] - u^r\|^2_R) \leq \beta . 
\end{equation}
\end{theorem}

\bpf
From Proposition~\ref{prop:feas} and the definition of $\gamma$ in Alg.~\ref{algo7} line 32, the optimal objective of~(\ref{eq:cvx_dist}) at consecutive discrete time-steps satisfies
\[
J^{j_{n+1},n+1}\leq J^{j_n,n} - \|x[n] - x^r\|_Q^2 - \|u[n] - u^r\|_R^2 + \beta .
\]
Summing this inequality over $n\geq 0$ yields~(\ref{eq:stagecost_bound}).
\epf


\section{Numerical results\protect\footnote{Simulation code implementing the DC tube MPC algorithms is available at \href{https://github.com/martindoff/DC-NN-MPC}{github.com/martindoff/DC-NN-MPC}.}}
\label{sec:results}

We consider the model of a planar vertical take-off and landing (PVTOL) aircraft \citep{mark}
\begin{equation*}
    \ddot{y} = (u_1 + g) \sin \alpha, \quad  \ddot{z} = (u_1 + g) \cos \alpha - g, \quad \ddot{\alpha} = u_2,
\end{equation*}
where $y, z, \alpha$ are the horizontal, vertical and angular displacement, $g=9.81$\,m\,s$^{-2}$ is acceleration due to gravity, and $u_1, u_2$ are proportional to the net thrust and torque acting on the aircraft. 
A discrete time model is obtained by Euler integration with time-step $\delta=0.5$\,s.
The system is constrained as follows: $\lvert u_1\rvert \leq 10$, $\lvert u_2\rvert\leq 10$, $\lvert \alpha\rvert \leq 3$, $\lvert \dot{\alpha} \rvert \leq 1$, $\lvert \dot{y}\rvert \leq 30$, $\lvert\dot{z}\rvert \leq 10$. 
The system state is defined as
$x = (\alpha, \dot{y}, \dot{z}, \dot{\alpha})$ and the
control input as $u = (u_1, u_2)$; the positions $y, z$ are recovered a
posteriori by integration. 
The cost weights are given as $Q= \text{diag}(10, 1, 1, 1)$, $R = \text{diag}(10^{-4}, 10^{-3})$, 
the terminal conditions are computed for $\lvert x - x^r\rvert \leq  (0.03, 1, 1, 0.1)$, 
$\lvert u - u^r \rvert \leq (1,1)$,
the initial conditions are $x[0] = ( 0.1, 0, 0, 0)$, and the system is to be stabilized around the hover state: $x^r = 0$, $u^r=0$.

\subsection{DC decomposition}

To exploit the linear components of the dynamics, we determine the DC decomposition of the nonconvex terms
\begin{equation}
\label{eq:noncvx}
f = \begin{bmatrix}
    \ddot{y} \\ \ddot{z}
\end{bmatrix} = \begin{bmatrix}
    (u_1 + g) \sin \alpha \\ (u_1 + g) \cos \alpha - g
\end{bmatrix}  .
\end{equation}
The remaining linear dynamics are accounted for in the MPC optimization through linear constraints.
We consider first the performance of the different DC decomposition methods outlined in Section \ref{sec:DC} for approximating the nonconvex part of the PVTOL dynamics in \eqref{eq:noncvx} as $f = g-h$ where $g, h$ are convex. We then exploit \textit{a priori} knowledge of the parameterization matrix $\Gamma$ in \eqref{eq:f_Q} to obtain the convex inequalities \eqref{eq:new_dyna}. The complexity of the DC decomposition in each case was chosen to balance the training effort required by each method.


\subsubsection{SOS-convex polynomials}

Polynomials of degree $2d = 6$ are fitted using $100,000$ samples generated from the PVTOL horizontal and vertical dynamics $\ddot{y}$ and  $\ddot{z}$. The SOS-convex polynomial DC decompositions are shown in Fig.~\ref{fig:DC1} and Fig.~\ref{fig:DC2}.  $100$ test samples were generated and the mean absolute error (MAE) between the polynomial approximation and the ground truth model~(\ref{eq:noncvx}) was $0.027$ for $\ddot{y}$ and $0.06$ for $\ddot{z}$.
\begin{figure}[htb]
    \centering
    \includegraphics[width=0.52\textwidth, trim={4mm 2mm 2mm 2mm}, clip]{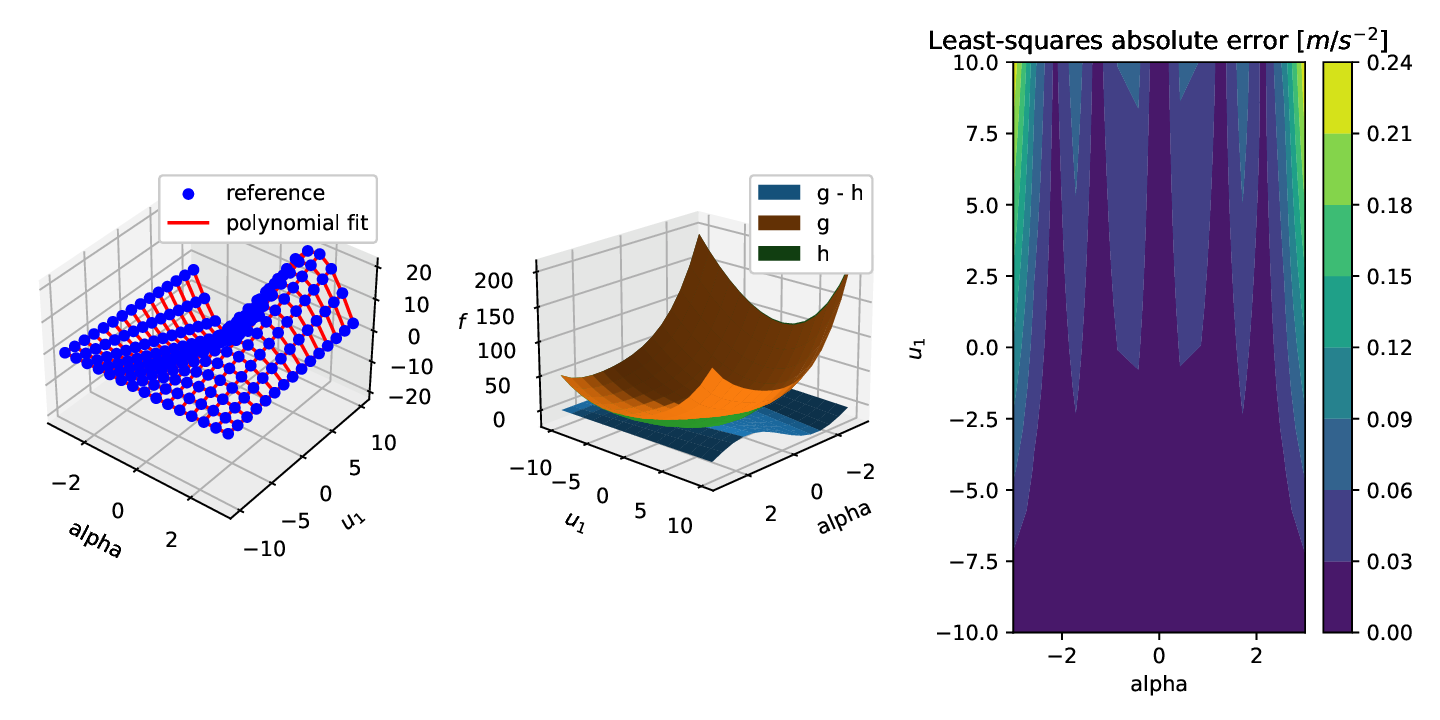}
    \caption{Approximation of PVTOL horizontal dynamics as a difference of SOS-convex polynomials.}
    \label{fig:DC1}
\end{figure}

\begin{figure}[htb]
    \centering
    \includegraphics[width=0.52\textwidth, trim={4mm 2mm 2mm 2mm}, clip]{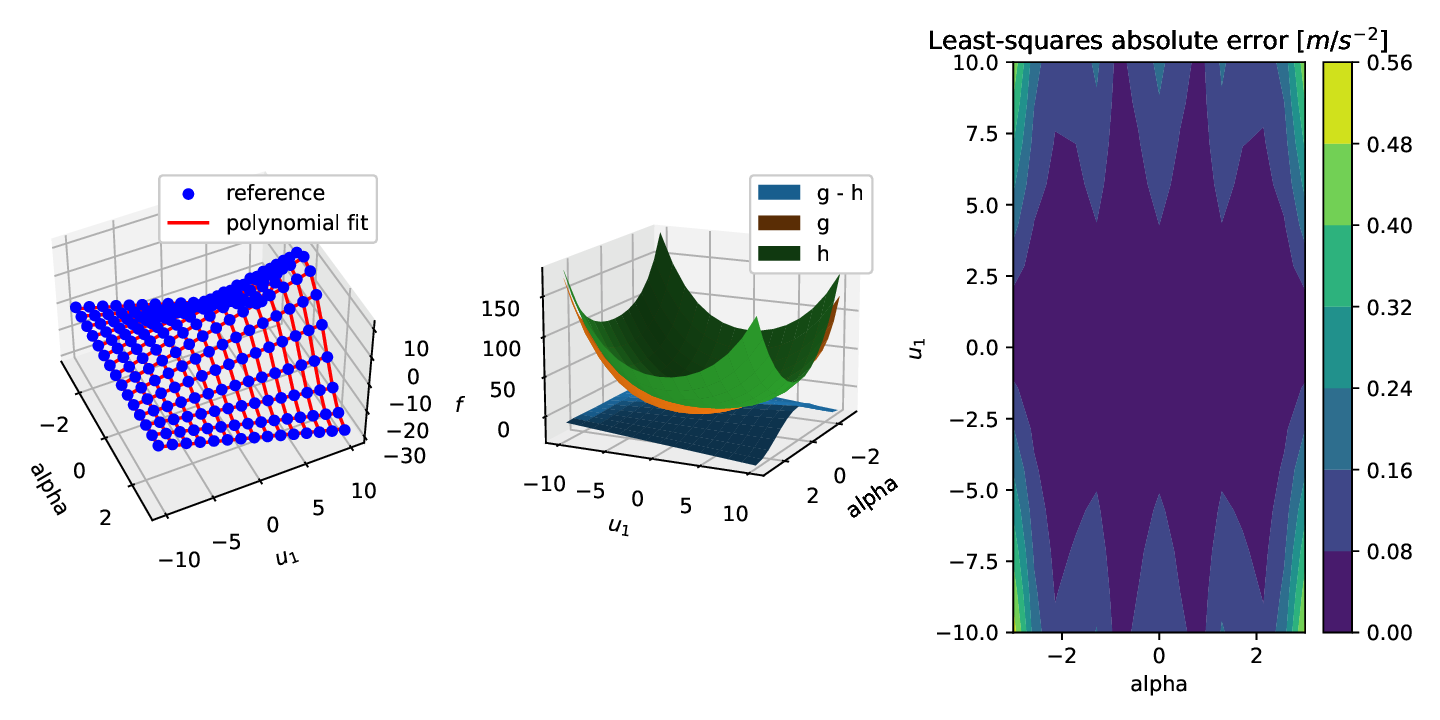}
    \caption{Approximation of PVTOL vertical dynamics as a difference of SOS-convex polynomials.}
    \label{fig:DC2}
\end{figure}

\subsubsection{Input-convex neural networks}
The DCNN architecture was computed in the present case study of the PVTOL aircraft model using the open-source neural-network library Keras \citep{chollet2015keras}. The network has a single hidden layer, $64$ units and ReLU activation, and was trained on $100,000$ samples generated from the model~(\ref{eq:noncvx}) using the RMSprop optimizer with MSE loss, batch size $32$, and $200$ epochs. An MAE of $0.08$ was obtained with 100 test samples (Fig.~\ref{fig:err}).  
The nonconvex dynamics (blue dots in Fig.~\ref{fig:DC0}) are approximated by the DCNN as a difference of two convex functions (green and orange dots in Fig.~\ref{fig:DC0}). 

\begin{figure}[htb]
    \centering
    \includegraphics[width=0.5\textwidth, trim={30mm 10mm 10mm 18mm}, clip]{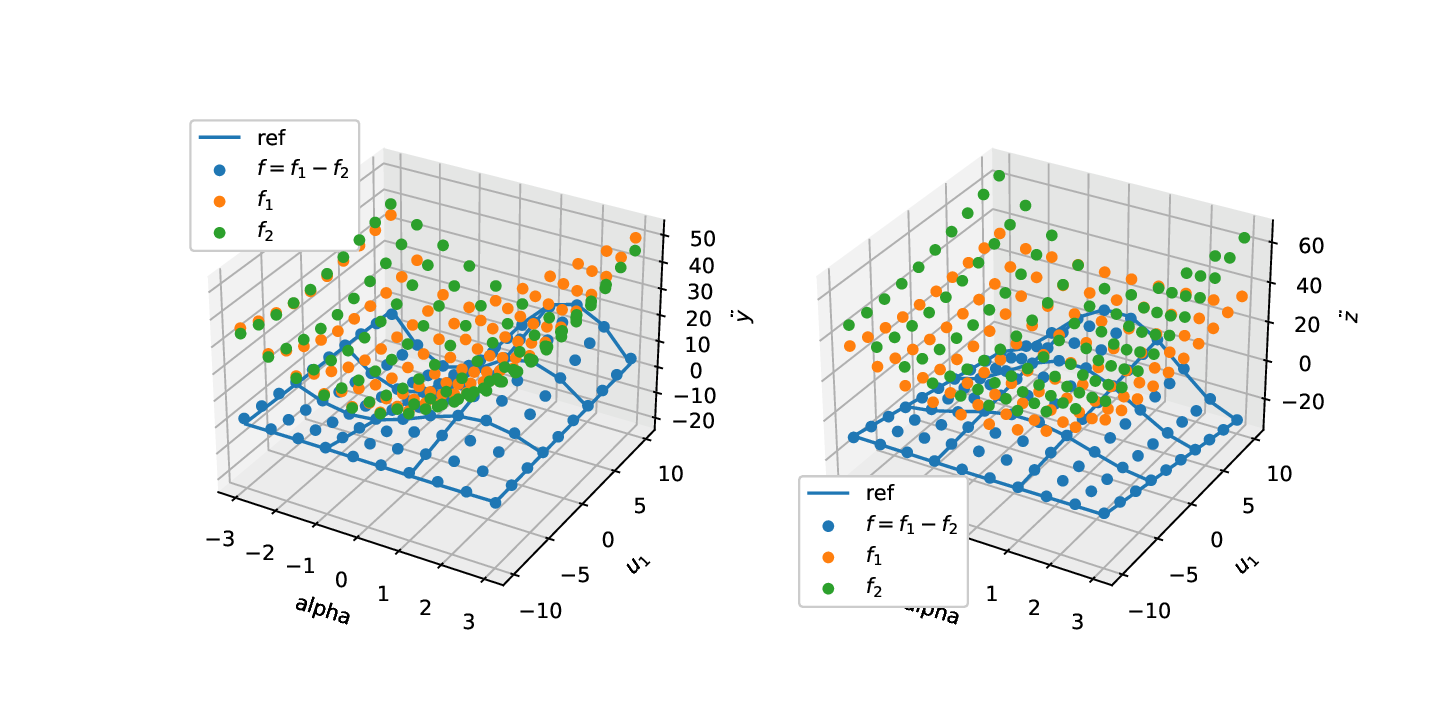}
    \caption{PVTOL model approximation using the DCNN.}
    \label{fig:DC0}
\end{figure}

\begin{figure}[htb]
    \centering
    \includegraphics[width=0.5\textwidth, trim={10mm 2mm 14mm 8mm}, clip]{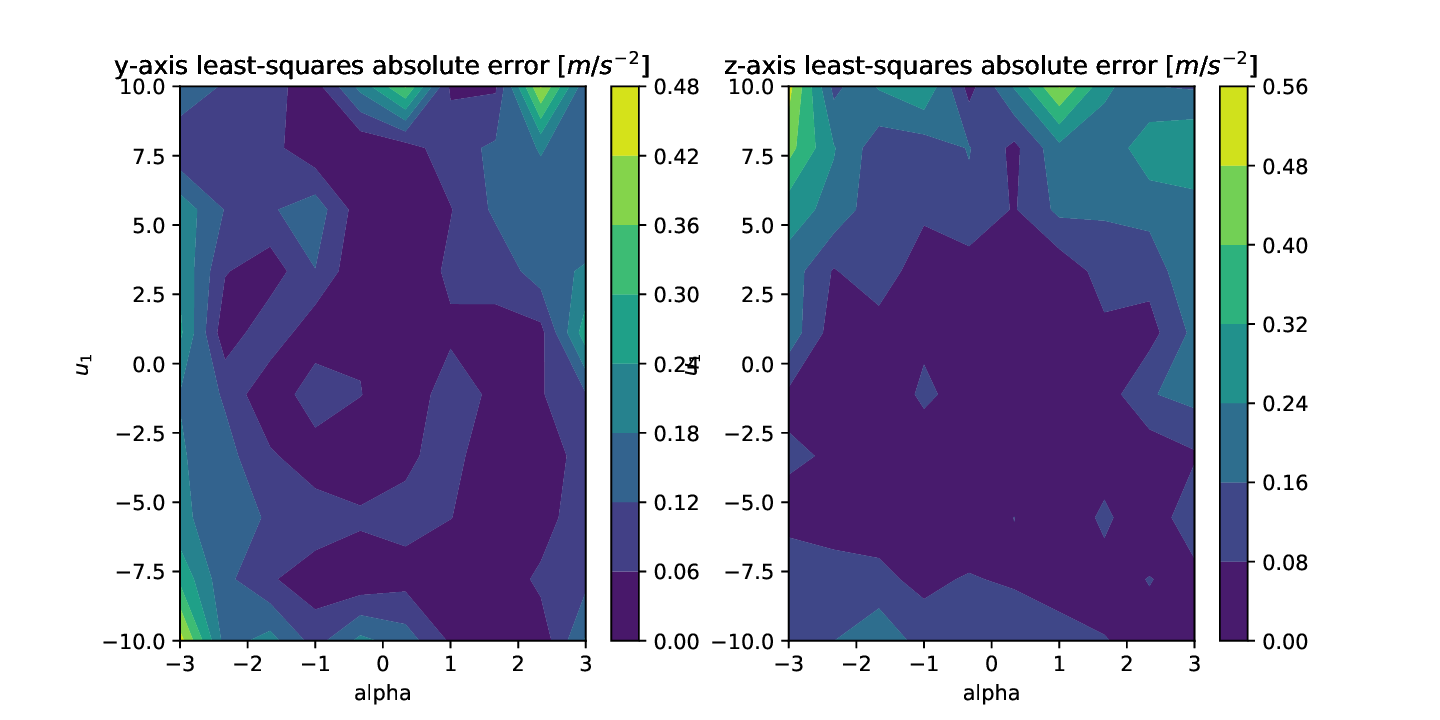}
    \caption{Model approximation error (MAE) for the DCNN.}
    \label{fig:err}
\end{figure}

\subsubsection{Radial basis functions}

A sum of $49$ radial basis functions with multiquadric kernel was fitted to $100,000$ samples of the model~(\ref{eq:noncvx}). This was achieved by training a single layer neural network with multiquadric activation (using RMSprop, MSE loss, batch size $64$ and $100$ epochs). The resulting accuracy (MAE) is $1.63$ on $100$ samples.

\subsection{Effect of DC decomposition on control performance}

We consider next the effects of different DC decompositions on closed-loop behaviour. The PVTOL aircraft problem was simulated with horizon $N=50$, sample interval $\delta=0.5$\,s,  elementwise bounds parameterization, and maximum number of iterations $\textit{itermax} =1$. For each configuration the closed loop response is shown in Fig.\ref{fig:tmpc_comp}, and the evolution of the MPC objective is shown in Fig.~\ref{fig:obj_comp}. The polynomial DC decomposition results in the fastest convergence with small steady state errors.

\begin{figure}[htb]
    \centering
    \includegraphics[width=0.5\textwidth, trim={3mm 4mm 2mm 4mm}, clip]{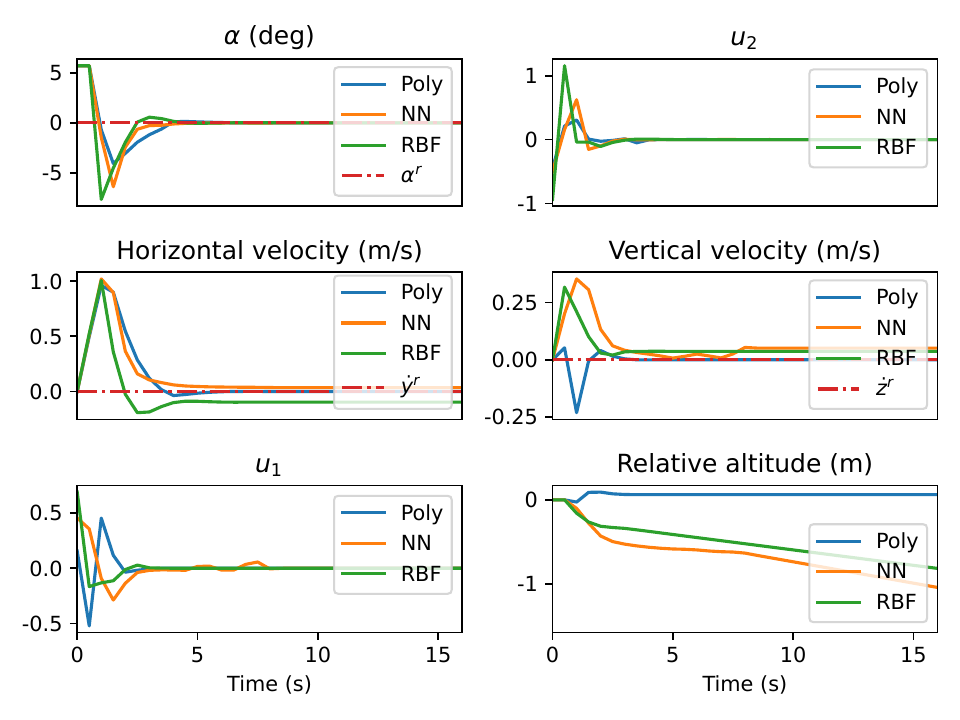}
    \caption{Closed-loop responses for each DC decomposition.}
    \label{fig:tmpc_comp}
\end{figure}

Setting the maximum number of iterations to $\textit{itermax} = 5$, we compare the evolution of the objective at the first time step in Figure \ref{fig:obj_comp0}. Again, the polynomial decomposition yields the fastest convergence.

\begin{figure}%
\begin{subfigure}[b]{0.25\textwidth}
         \includegraphics[width=1\textwidth, trim={3mm 2mm 10mm 12mm}, clip]{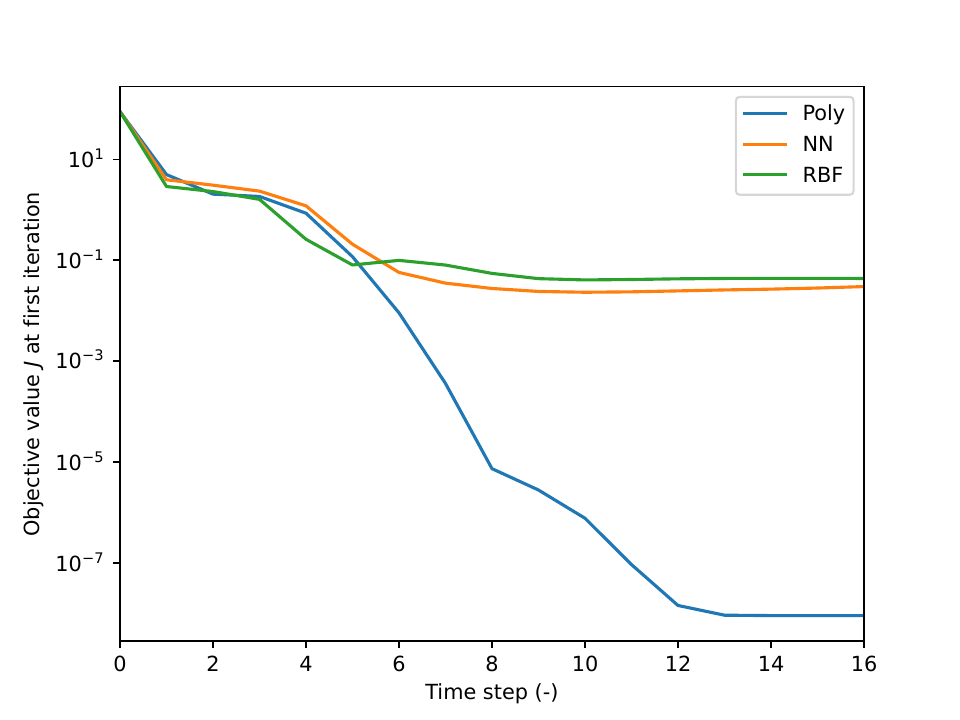}
    \caption{}
    \label{fig:obj_comp}
\end{subfigure}%
\begin{subfigure}[b]{0.25\textwidth}
    \includegraphics[width=1\textwidth, trim={3mm 2mm 10mm 12mm}, clip]{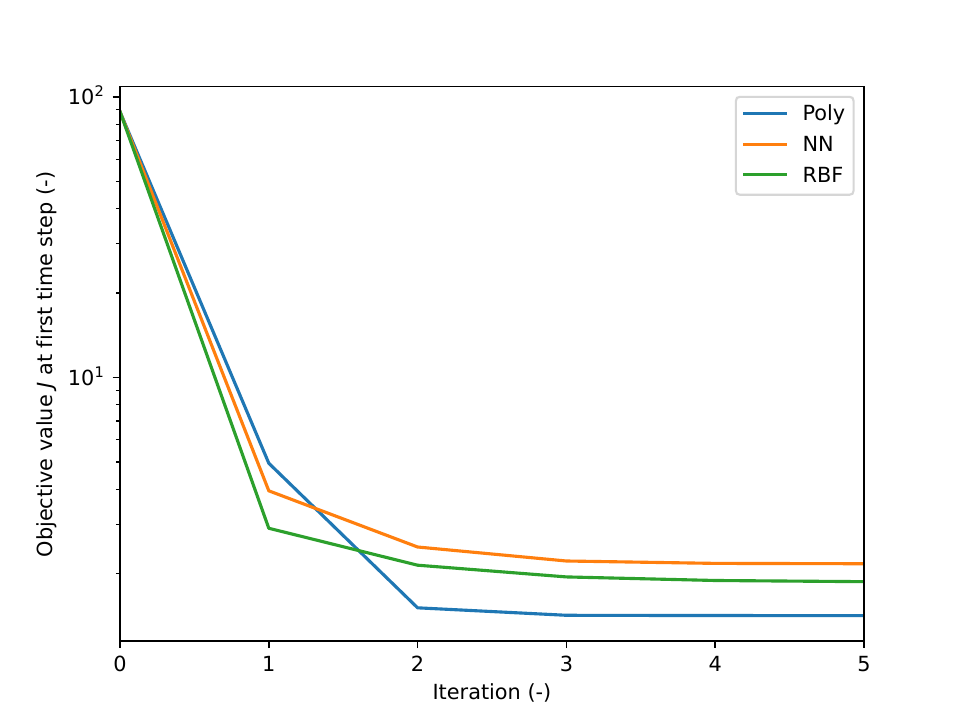}
    \caption{}
    \label{fig:obj_comp0}
\end{subfigure}
\caption{MPC objective values for each DC decomposition.}
    \label{fig:obj_comp2}
\end{figure}

Simulations were run with increasing horizon length $N$. The average computation time per iteration for each decomposition is shown in Fig.~\ref{fig:time_comp}. The polynomial and neural network DC decompositions allow the fastest MPC optimization, while the MPC optimization with RBF DC decomposition is significantly slower. We note that (i) CVXPY was used as the interface for the MOSEK solver, but directly calling specialized solvers might yield better performance (e.g.~an ADMM solver may allow an order of magnitude speed up~\citep{doff2022predictive}); (ii) as discussed in \citet{doff2023robust}, the polynomial approach relies on a low-order quadratic approximation at each point of the predicted trajectory in order to represent the polynomial constraint in CVXPY more efficiently; (iii) performance depends on the choice of hyperparameters for a given decomposition. 

\begin{figure}[htb]
    \centering
    \includegraphics[width=0.5\textwidth, trim={10mm 2mm 12mm 12mm}, clip]{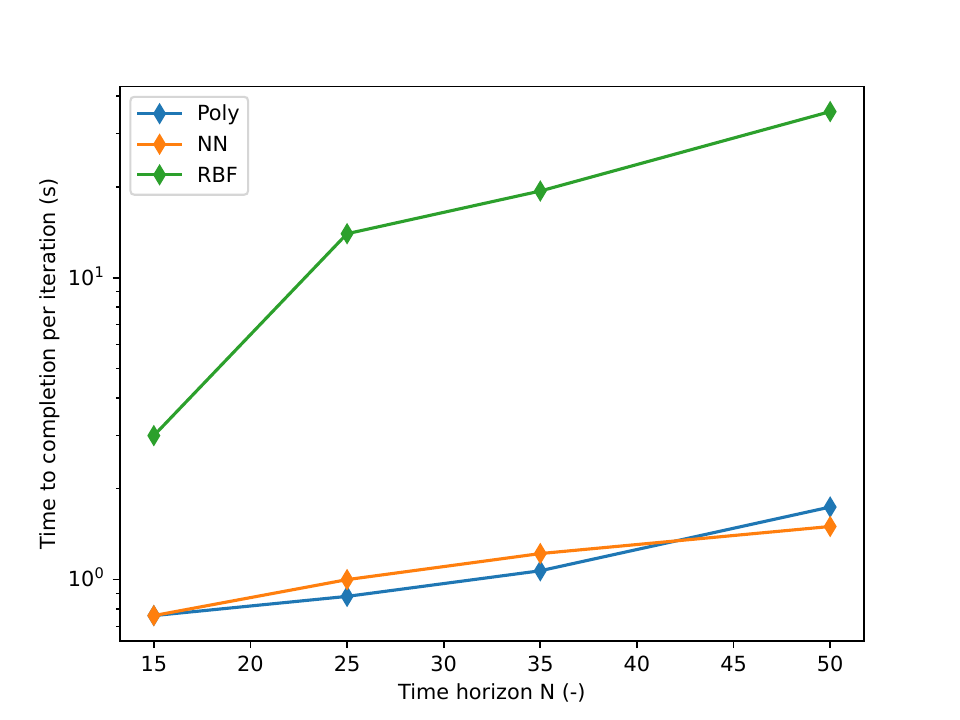}
    \caption{Variation of computation with problem size.}
    \label{fig:time_comp}
\end{figure}

We conclude that in this application polynomial decomposition allows faster MPC optimization, faster convergence to the target state and reduced steady state errors.  

\subsection{Reducing curvature}

Controller performance is affected by errors introduced when linearizing the concave part of the DC decomposition.  It is therefore advantageous to use a DC decomposition that minimizes the curvature of $h$.  To demonstrate the benefit of reducing curvature, we retrain the DCNN (ReLU activation, $1$ hidden layer, $8$ units) with MSE loss augmented by a finite-difference curvature estimate:
 \[
\mathcal{L} = \mathrm{MSE} \;+\; \lambda \sum_{i=1}^{2}
  \Bigl( h(x + d_{i}e_{i}) - 2h(x) + h(x - d_{i}e_{i}) \Bigr)
\]
where $x = (\alpha,u_1)$,  $d = (0.2, 1)$ and $e_1,e_2$ are basis vectors of $\RR^{2}$.  
For a twice-differentiable function and sufficiently small $d_i$, this finite difference approximates \(d_i^2[\nabla^2h(x)]_{ii}\); for the ReLU DCNN it is used directly as a discrete curvature measure.
The network is trained with $\lambda = 0$ (unpenalized) and $\lambda = 1.5$ using $10^5$ samples of the PVTOL model (batch size $32$, $200$ epochs). 
%


The distributions of linearization errors measured over tube-sized steps for both models (Fig.~\ref{fig:lin_err}) show that the penalized model is effective at reducing linearization errors. 
Penalizing curvature during training results in a controller with a larger region of attraction (Fig.~\ref{fig:roa}), confirming that smaller linearization errors reduce conservativeness. Table~\ref{tab:bench} shows the effect on controller performance of training with and without a penalty on curvature, using simulations over $15$ time steps at $25$ initial conditions drawn from the common feasible region, with the convex-concave procedure terminated using $\textit{tolerance} = 0.05$.
Therefore, reducing the model curvature accelerates convergence,
tightens the tube width and increases constraint margins (i.e.~the tightest slack between the tube and state constraints over the horizon).

\renewcommand{\figurename}{Tab.}
\newcounter{oldfigure}
\setcounter{oldfigure}{\value{figure}}
\setcounter{figure}{\value{table}}
\addtocounter{table}{1}
\begin{figure}
\centering
\begin{tabular}{lcccc}
 & iters & tube width & margin & $J_{\mathrm{cl}}$ \\
\hline
\multicolumn{1}{l|}{$\lambda = 0$}        & $2.48$ & $2.03$ & $0.636$ &  $255.1$ \rule{0pt}{11pt}\\
\multicolumn{1}{l|}{$\lambda=1.5$} & $2.21$ & $1.53$ & $0.693$  & $238.6$ \\
\multicolumn{1}{l|}{change \smash{($\%$)}} & $-10.7$ & $-24.7$ & $+9.1$  & $-6.4$
\end{tabular}
\caption{Effect of penalizing curvature on closed loop performance, averaged over $25$ initial conditions.}
\label{tab:bench}
\end{figure}
\setcounter{figure}{\value{oldfigure}}
\renewcommand{\figurename}{Fig.}

 \begin{figure}[htb]
   \includegraphics[scale=0.4, trim={2mm 4mm 2mm 2mm}, clip]{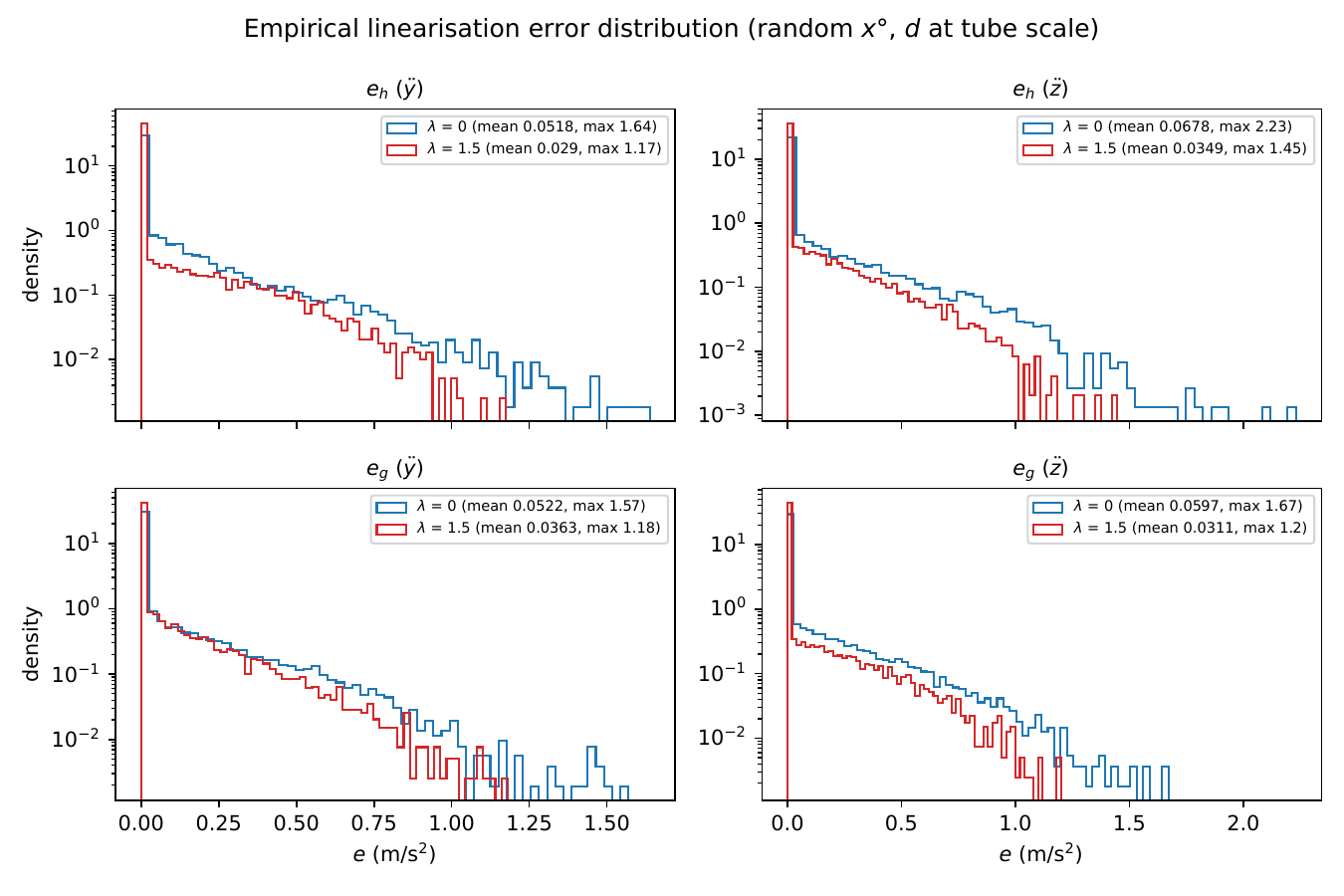}
   \caption{Distribution of the linearization errors in $g$ and $h$ without and with curvature penalization.}
   \label{fig:lin_err}
 \end{figure}

\begin{figure}[htb]
   \includegraphics[scale=0.4, trim={3mm 4mm 2mm 2mm}, clip]{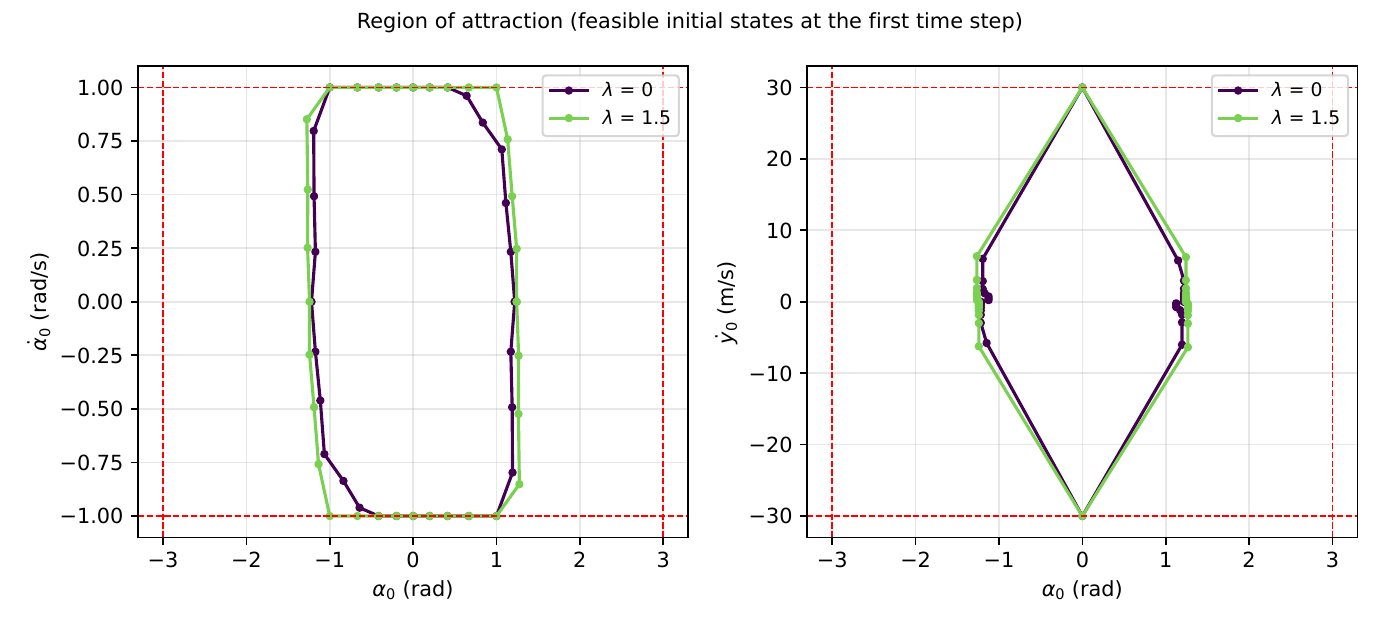}
   \caption{Two-dimensional slices of the region of attraction of the closed loop system, without and with curvature penalization.}
   \label{fig:roa}
 \end{figure}

\subsection{Effect of additive disturbance}

The PVTOL aircraft model was modified to include uniformly distributed $\mathcal{U}(-0.1, 0.1)$ additive disturbances acting on the horizontal and vertical velocity dynamics. The closed loop responses (Fig.~\ref{fig:dist}) show the disturbance rejection capabilities of the control law of Alg.~\ref{algo7} (with neural-network DC decomposition, prediction horizon $N=50$, sample interval $\delta=0.5$\,s).
%
The controller maintains good disturbance rejection (Fig.~\ref{fig:dist3}) when the disturbance bounds are increased to account for bounds on DCNN modelling errors, as discussed in Section \ref{sec:dist}. 

\begin{figure}[htb]
    \includegraphics[scale=0.55, trim={2mm 4mm 2mm 4mm}, clip]{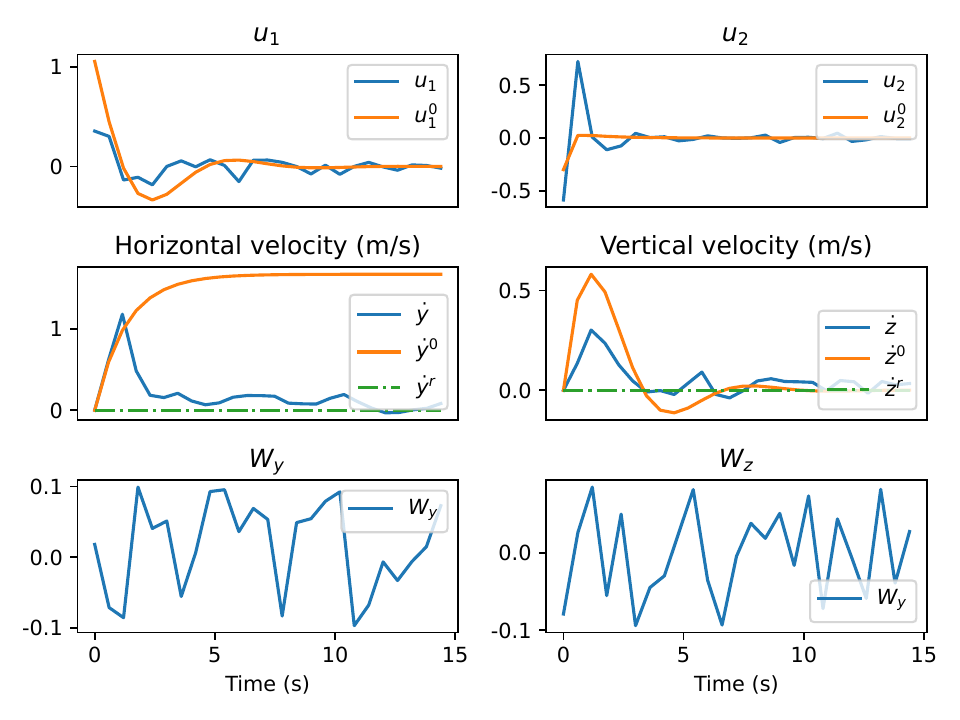}
    \caption{Closed-loop response with additive disturbance.}
    \label{fig:dist}
\end{figure}

\begin{figure}[htb]
    \includegraphics[scale=0.55, trim={2mm 4mm 2mm 4mm}, clip]{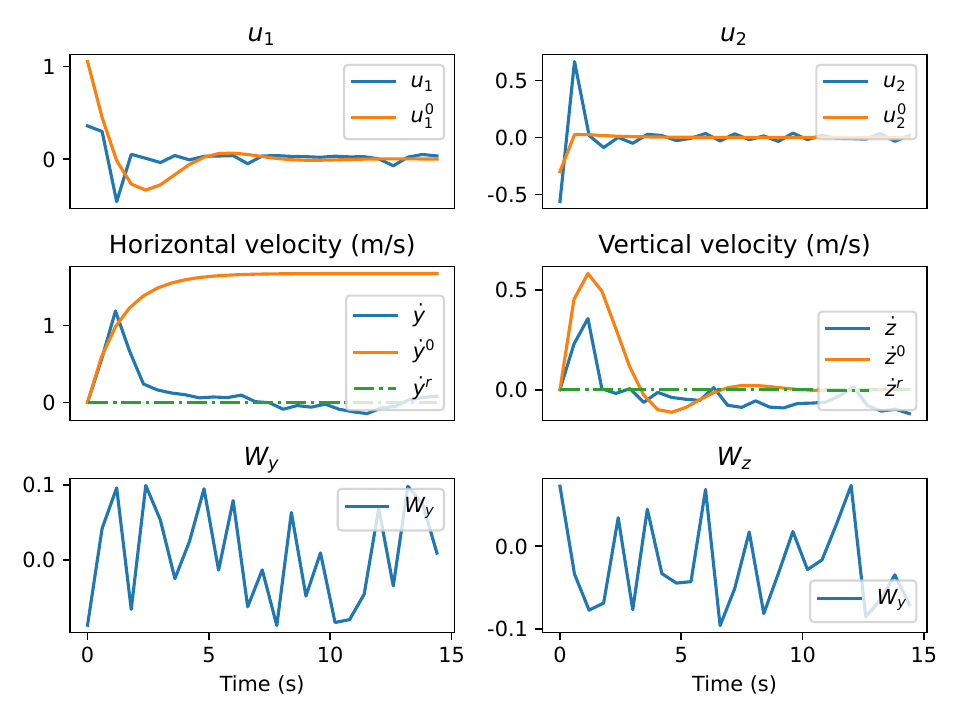}
    \caption{Closed-loop response with additive disturbance, robust to neural network modelling errors.}
    \label{fig:dist3}
\end{figure}

To demonstrate recovery from infeasible instances of problem (\ref{eq:cvx_dist}) using the line search in Alg.\ref{algo7} lines 17-25, the disturbance bounds were increased by a factor of $10$ and the line search parameter chosen as $\rho = 0.2$. For the same initial condition and controller parameters as in the responses of Fig.~\ref{fig:dist3}, the optimization (\ref{eq:cvx_dist}) was found to be infeasible at $8$ iterations.
Table~\ref{tab:lsearch} shows the time step $n$, iteration number $j$, and the minimum constraint relaxation required to obtain a feasible problem, as indicated by the value of a slack variable $\epsilon$ added to each of the constraints defining $\X,\U$. Here $\epsilon_{l}$ is the relaxation required after $l$ line search iterations, and
in each case feasibility was recovered after at most $l = 2$ line search iterations ($\epsilon < 10^{-6}$ indicates a feasible problem).

Under stringent constraints on computation it may be necessary to use Algorithm~\ref{algo7} with a single linearization per time-step ($\textit{itermax} = 1,\textit{itermax\_ls} = 0$). The closed loop responses for this case with unit disturbance bounds maintain good disturbance rejection (Fig.~\ref{fig:lsearchi1}), despite the order-of-magnitude increase in disturbance bounds.

\renewcommand{\figurename}{Tab.}
\setcounter{oldfigure}{\value{figure}}
\setcounter{figure}{\value{table}}
\addtocounter{table}{1}
\begin{figure}
    \centerline{\begin{tabular}{ccccc}
         $n$ & $j$ & $l$ & $\epsilon_0$ & $\epsilon_l$ \\ 
         \hline 
         $2$ & $2$ & $2$ & $5.0250 \times 10^{-4}$ & $7.7796 \times 10^{-8}$\rule{0pt}{11pt} \\ 
         $10$ & $3$ & $1$ & $1.3031 \times 10^{-4}$ & $6.5674 \times 10^{-9}$ \\ 
         $10$ & $4$ & $1$ & $1.3620 \times 10^{-4}$ & $1.8534 \times 10^{-8}$ \\ 
         $10$ & $5$ & $1$ & $1.2475 \times 10^{-4}$ & $7.7934 \times 10^{-8}$ \\ 
         $13$ & $1$ & $1$ & $4.3111 \times 10^{-4}$ & $1.5087 \times 10^{-7}$ \\ 
         $13$ & $4$ & $1$ & $1.4110 \times 10^{-2}$ & $1.4308 \times 10^{-7}$ \\ 
         $13$ & $5$ & $1$ & $1.4315 \times 10^{-2}$ & $2.9371 \times 10^{-8}$ \\ 
         $22$ & $5$ & $1$ & $4.1392 \times 10^{-4}$ & $1.2096 \times 10^{-7}$ 
    \end{tabular}}
    \caption{Infeasible iterations and constraint relaxations $\epsilon_0,\epsilon_l$.}
    \label{tab:lsearch}
\end{figure}
\setcounter{figure}{\value{oldfigure}}
\renewcommand{\figurename}{Fig.}

\begin{figure}[htb]
    \includegraphics[scale=0.4, trim={2mm 4mm 2mm 4mm}, clip]{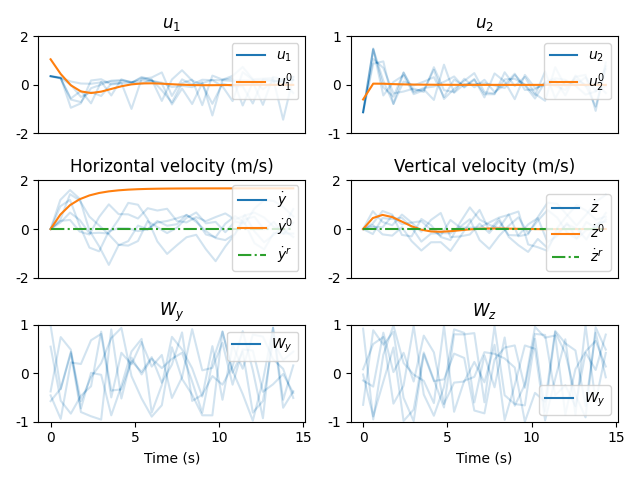}
    \caption{Closed-loop responses with $10\times$ disturbance bounds and one iteration per time-step ($\textit{itermax}  = 1,\textit{itermax\_ls}  = 0$).\!}
    \label{fig:lsearchi1}
\end{figure}

\section{Conclusions}

This work proposes a computationally tractable robust MPC algorithm for nonlinear systems with continuous dynamics. The method combines numerical difference-of-convex (DC) decompositions of the system model, successive convexification via partial linearization and robust stabilization via a tube MPC scheme accounting for model approximation and linearization errors. The DC form of the proposed MPC optimization permits the use of a single linearization iteration per time-step without compromising feasibility, stability and asymptotic performance guarantees. 


Preliminary results using the convex optimization package CVXPY are promising and suggest excellent scalability if specific solvers are developed. 
In particular, the current approach enumerates tube vertices, which may result in large numbers of constraints, but the critical vertices can often be determined prior to passing constraints to the solver. 
This approach can be applied to all of the constraints that appear in DC form if monotonicity is enforced during training and exploited to determine the potentially active tube vertices at each iteration of the convex-concave procedure.
Other research directions involve applying the successive convexification approach to stochastic MPC with scenario optimization and to differentiable MPC with online model adaptation.  

\bibliography{biblio} 
\bibliographystyle{unsrtnat}

\end{document}